\documentclass[11pt, reqno]{amsart}

\usepackage{fullpage}
\usepackage{amsmath}
\usepackage{amsthm}
\usepackage{amsfonts}
\usepackage{amssymb}

\newcommand{\mb}{\mathbf}
\newcommand{\mc}{\mathcal}
\renewcommand{\Re}{\mathrm{Re}\,}

\newcommand{\rg}{\mathrm{rg}\,}

\newtheorem{lemma}{Lemma}[section]
\newtheorem{proposition}[lemma]{Proposition}
\newtheorem{theorem}[lemma]{Theorem}
\newtheorem{corollary}[lemma]{Corollary}
\theoremstyle{remark}

\theoremstyle{definition}
\newtheorem{definition}[lemma]{Definition}

\newcommand{\R}{\mathbb{R}}

\title{Stable self--similar blowup in energy supercritical Yang--Mills theory}

\author{Roland Donninger}
\address{\'Ecole Polytechnique F\'ed\'erale de Lausanne, 
MA B1 487, Station 8, CH-1015 Lausanne, Switzerland }
\email{roland.donninger@epfl.ch}

\begin{document}

\maketitle

\begin{abstract}
We consider the Cauchy problem for an energy supercritical nonlinear wave equation
that arises in $(1+5)$--dimensional Yang--Mills theory.
A certain self--similar solution $W_0$ of this model
is conjectured to act as an 
attractor for generic large data evolutions.
Assuming mode stability of $W_0$, we prove a weak version of this conjecture, namely
that the self--similar solution $W_0$ is (nonlinearly) stable.
Phrased differently, we prove that mode stability of $W_0$ implies its nonlinear stability.
The fact that this statement is not vacuous follows from careful numerical work by Bizo\'n and
Chmaj that
verifies the mode stability of $W_0$ beyond reasonable doubt.
\end{abstract}

\section{Introduction}

This paper is concerned with the study of the Cauchy problem for the semilinear wave equation
\begin{equation}
\label{eq:main}
\psi_{tt}(t,r)-\psi_{rr}(t,r)-\frac{d-3}{r}\psi_r(t,r)+\frac{d-2}{r^2}F(\psi(t,r))=0 
\end{equation}
for $d=5$ where $F(\psi)=\psi(\psi+1)(\psi+2)$ and $r=|x|$, $x\in \R^d$.
Eq.~\eqref{eq:main} arises in $SO(d)$--equivariant Yang--Mills theory, see \cite{D82}, \cite{B02}
for a derivation.
Historically, the introduction of nonabelian gauge theory by Yang and Mills was fundamental for 
the development of the standard model of particle physics, see \cite{T05}.
Apart from that, the Yang--Mills model as a classical field theory attracted a lot of interest too, 
cf.~\cite{A79}.
Furthermore, Yang--Mills equations 
have been proposed
as toy models for Einstein's equations of general relativity, see e.g.~\cite{EM82a}, \cite{GM07}.
Especially in this context the development of singularities in finite time is of interest.

Eq.~\eqref{eq:main} is invariant under scaling and it admits a conserved energy
\[ E(\psi(t,\cdot),\psi_t(t,\cdot))=\int_0^\infty \left [ \psi_t(t,r)^2+\psi_r(t,r)^2+\frac{d-2}{2r^2}\psi(t,r)^2(\psi(t,r)+2)^2 \right ]r^{d-3}dr \]
which scales like 
\[ E(\psi^\lambda(\lambda t,\cdot),\psi_t^\lambda(\lambda t,\cdot))=\lambda^{d-4}E(\psi(t,\cdot),\psi_t(t,\cdot)) \] 
where $\psi^\lambda(t,r):=\psi(t/\lambda, r/\lambda)$, $\lambda>0$.
This shows that Eq.~\eqref{eq:main} is energy subcritical in the physical dimension $d=3$ whereas
it is critical for $d=4$ and supercritical for $d\geq 5$.
According to the usual blowup heuristics, where energy conservation prevents the solution
from shrinking in the energy subcritical case, one expects global existence for $d=3$.
Indeed, Eardley and Moncrief \cite{EM82a}, \cite{EM82b} considered this problem (without symmetry assumptions)
and proved global existence for data in suitable Sobolev spaces without restriction on size.
This classical result was strengthened by Klainerman and Machedon \cite{KM95} who 
lowered the required degree of regularity of the data and provided a different approach to the problem, see also
\cite{K97}, \cite{T03}.

Yang--Mills fields in dimension $d=4$ attracted a lot of interest in the recent past due 
to new developments in the study of energy critical wave equations.
Local well-posedness under minimal
regularity assumptions was considered by Klainerman, Tataru \cite{KT99} and global existence
for small data was proved by Sterbenz \cite{S07}.
In the critical dimension $d=4$, Eq.~\eqref{eq:main} admits a static, finite energy solution which
is known as the \emph{instanton}.
This indicates the existence of more complex dynamics than in the case $d=3$.
In particular, the energy of the instanton represents a threshold for global existence.
Indeed, C\^ote, Kenig and Merle \cite{CKM08} proved global existence and scattering 
(either to zero or to a rescaling of the instanton)
for data with energy
below (or equal to) the energy of the instanton.
On the other hand, it is known that the Yang--Mills system in the critical dimension $d=4$ can develop
singularities in finite time.
This has been conjectured by Bizo\'n \cite{B02} and demonstrated numerically in \cite{B02}, \cite{LS02}. 
Furthermore, the blowup rate was derived by Bizo\'n, Ovchinnikov and Sigal \cite{BOS04}.
The existence of blowup solutions was proved rigorously by Krieger, Schlag, Tataru 
\cite{KST09} as well as Rapha\"el and Rodnianski \cite{RR09} who also obtained the stable blowup rate.

In the supercritical dimension $d=5$, which shall concern us here, much less is known. 
In general, the study of energy supercritical wave equations is still only at the beginning, see, however, 
e.g.~\cite{KM11a}, \cite{KM11}, \cite{KV11a}, \cite{KV11}, \cite{B11}, \cite{B10}, \cite{D11} for recent progress. 
It is clear that energy supercritical problems will have to play a prominent role in the future development of the field, 
not only because of their relevance in physics which can hardly be overstressed.
In particular, much work remains to be done in order to improve our understanding of \emph{large} 
solutions which at the moment is mostly confined to
the construction of self--similar solutions by solving a corresponding
elliptic ODE problem, a procedure which is insensitive to the criticality class of the problem, 
see e.g.~\cite{S88}, \cite{CST98}, \cite{B00}, \cite{BMW07}.
In the case of the Yang--Mills field in $d=5$ 
one has global existence \cite{S10} for small data, see also \cite{KS05} for $d\geq 6$, 
whereas for large data finite time blowup is
possible \cite{CST98}.
The failure of global existence has been demonstrated by constructing self--similar solutions to Eq.~\eqref{eq:main}.
In fact, Bizo\'n \cite{B02} showed that there exists a countable family of self--similar solutions. 
Furthermore, he was even able to find an explicit expression for the ``ground state'' of this family which
we denote by $\psi^T$ and it reads 
\[ \psi^T(t,r)=W_0(\tfrac{r}{T-t})-1 \]
where $T>0$ is a constant (the blowup time) and 
\[ W_0(\rho)=\frac{1-\rho^2}{1+\frac35 \rho^2}. \]
The relevance of such an explicit solution for understanding the dynamics of the equation
depends on its stability.
In other words, the important question is: does the blowup described by $\psi^T$ occur for a ``large'' set of initial data?
Numerical simulations \cite{BC05} indicate that this is indeed the case.
Moreover, it appears that the blowup via $\psi^T$ is ``generic'', i.e., sufficiently large,
``randomly'' chosen initial data lead to an evolution which asymptotically (as $t \to T-$) converges
to $\psi^T$.
Consequently, in the present paper we study the stability of $\psi^T$ and obtain the following result,
see Theorem \ref{thm:main} below for the precise statement.

\begin{theorem}[main result, qualitative version]
\label{thm:mainqual}
Suppose $\psi^T$ is mode stable. Then there exists an open set (in a topology strictly stronger
than the energy) of initial data for Eq.~\eqref{eq:main} such that the corresponding time evolution approaches $\psi^T$
and blows up.
\end{theorem}

We remark that the technical assumption we have to make, the mode stability of $\psi^T$, is equivalent
to a certain spectral property of a (nonself--adjoint) second order ordinary differential operator, see
below.
Unfortunately, we are not able to prove this spectral property rigorously but it has been verified
numerically beyond reasonable doubt \cite{BC05}.
In the following we will comment on this issue in more detail.

\subsection{The mode stability problem}
A first step in the stability analysis of the self--similar solution $\psi^T$ is to look for 
unstable mode solutions. To this end we insert the ansatz $\psi=\psi^T+\varphi$ into the 
Yang--Mills equation \eqref{eq:main} with $d=5$ and linearize by dropping all nonlinear terms in $\varphi$.
This yields the evolution equation
\begin{equation}
\label{eq:linear}
\varphi_{tt}-\varphi_{tt}-\tfrac{2}{r^2}\varphi_r+\tfrac{3}{r^2}F'(W_0(\tfrac{r}{T-t})-1)\varphi=0 
\end{equation}
for the perturbation $\varphi$.
In order to obtain a time independent potential we introduce $\rho=\frac{r}{T-t}$ as a new variable
and restrict ourselves to the backward lightcone of the blowup point $(T,0)$ by assuming
$\rho \in [0,1]$.
A particularly convenient choice for a new time variable is 
$\tau=-\log(T-t)$.
The coordinates $(\tau,\rho)$ are sometimes referred to as ``similarity variables'' and they are
frequently used when dealing with self--similar solutions for nonlinear wave equations, see e.g.~\cite{D10}, \cite{D11}, 
\cite{DS12}, \cite{MZ03}, \cite{MZ05}, \cite{MZ07}.
Note that the blowup takes place as $\tau \to \infty$ and thus, we are effectively dealing with
an asymptotic stability problem.
By setting $\varphi(t,r)=\phi(-\log(T-t),\frac{r}{T-t})$, Eq.~\eqref{eq:linear} transforms
into
\begin{equation}
\label{eq:linearcss}
\phi_{\tau \tau}+\phi_\tau+2\rho \phi_{\tau \rho}-(1-\rho^2)\left [\phi_{\rho \rho}
+\tfrac{2}{\rho^2}\phi_\rho \right ]+\tfrac{3}{\rho^2}F'(W_0(\rho)-1)\phi=0.
\end{equation}
A solution $\phi_\lambda$ of Eq.~\eqref{eq:linearcss} of the form 
$\phi_\lambda(\tau,\rho)=e^{\lambda \tau}u_\lambda(\rho)$ for $\lambda \in \mathbb{C}$ and
a nonzero function $u_\lambda \in C^\infty[0,1]$ is called a \emph{mode solution}.
It will become clear below why we can restrict ourselves to smooth $u_\lambda$.
Furthermore, we say that $\lambda$ is an \emph{eigenvalue} (of $\psi^T$) if there exists a corresponding mode
solution $\phi_\lambda$.
For obvious reasons a mode solution $\phi_\lambda$ (or an eigenvalue $\lambda$) 
is called \emph{stable} if 
$\mathrm{Re}\lambda<0$ and \emph{unstable} otherwise.
At this point it is worth emphasizing that a priori the nonexistence of unstable mode solutions
is \emph{neither necessary nor sufficient} for the (nonlinear) stability of $\psi^T$.
However, the nonexistence of unstable mode solutions is obviously necessary for the \emph{linear} stability
of $\psi^T$ and, as we will prove in this paper, it is even \emph{sufficient} for the \emph{nonlinear} 
stability of $\psi^T$ (see Theorem \ref{thm:main} for the precise statement).
As a consequence, it is crucial to understand mode solutions and, by inserting the ansatz
$\phi_\lambda(\tau,\rho)=e^{\lambda \tau}u_\lambda(\rho)$ into Eq.~\eqref{eq:linearcss}, this problem reduces
to the ODE
\begin{equation}
\label{eq:modstab}
-(1-\rho^2)\left [u_\lambda''+\tfrac{2}{\rho}u_\lambda' \right ]+2\lambda \rho u_\lambda'
+\lambda(\lambda+1)u_\lambda+\tfrac{3}{\rho^2}F'(W_0(\rho)-1)u_\lambda=0.
\end{equation}
As a matter of fact, there exists an unstable mode solution for $\lambda=1$ given
by $u_1(\rho)=\rho W_0'(\rho)$.
However, it turns out that this is a symmetry mode, i.e., it stems from the time translation symmetry
of Eq.~\eqref{eq:main} and does not count as a ``real'' instability (this will become much clearer
in Section \ref{sec:globalex} below, see in particular Lemma \ref{lem:U}).
Consequently, we define

\begin{definition}
\label{def:modstab}
The solution $\psi^T$ is said to be \emph{mode stable} iff $u_1(\rho)=\rho W_0'(\rho)$ is the only solution 
of Eq.~\eqref{eq:modstab} in $C^\infty[0,1]$ with $\mathrm{Re}\lambda\geq 0$.
\end{definition}

Unfortunately, it appears to be extremely difficult to exclude unstable mode solutions
(one has to bear in mind that the problem is nonself-adjoint; so in principle there could
be unstable eigenvalues with nonzero imaginary parts).
If $\mathrm{Re}\lambda\geq 1$, the problem is fairly easy since one can resort to Sturm--Liouville 
oscillation theory \cite{B02} and it is well--known that there do not exist unstable eigenvalues
$\lambda$ with $\Re\lambda \geq 1$ apart from the aforementioned symmetry mode.
However, the domain $0 \leq \mathrm{Re}\lambda <1$ seems to be very challenging.
The fact that exactly the same problem occurs in the study of energy supercritical wave maps 
\cite{DSA11}, \cite{D11} underlines the importance of having a general approach to that kind of 
nonself--adjoint spectral problems.
For the moment, however, this question remains open.
On the other hand, there are very reliable numerical techniques to study boundary value problems
of the type \eqref{eq:modstab}.
As a consequence, the mode stability of $\psi^T$ has been established numerically beyond 
reasonable doubt \cite{BC05}.
In addition, we provide a new result (see Lemma \ref{lem:nospec} below) which excludes unstable
eigenvalues that are far away from the real axis.
This puts the available numerics on an even stronger footing.

\subsection{The main result}
With these technical preparations at hand we can formulate our main result.
To begin with, we define a norm
\begin{align*} \|(f,g)\|_{\mc E(R)}^2:=&\int_0^R \left |rf'''(r)+6f''(r)+\tfrac{3}{r}f'(r)-\tfrac{3}{r^2}f(r) \right |^2 dr \\
&+\int_0^R \left |rg''(r)+5g'(r)+\tfrac{3}{r}g(r)\right |^2 dr
\end{align*}
on $\tilde{\mc E}(R):=\{(f,g)\in C^3[0,R]\times C^2[0,R]: f(0)=f'(0)=g(0)=0\}$. 
It is easily seen that $\|\cdot\|_{\mc E(R)}$ is indeed a norm on this space and we denote
by $\mc E(R)$ the completion of $\tilde {\mc E}(R)$ with respect to $\|\cdot\|_{\mc E(R)}$.
Our main result is the following.

\begin{theorem}[stable self--similar blowup]
\label{thm:main}
Assume $\psi^T$ to be mode stable.
Let $\varepsilon>0$ and suppose we are given initial data $(f,g) \in \mc E(\frac32)$ such that
\[ \|(f,g)-(\psi^1(0,\cdot),\psi^1_t(0,\cdot)\|_{\mc E(\frac32)} \]
is sufficiently small.
Then there exists a unique solution $\psi$ of Eq.~\eqref{eq:main} with $d=5$ satisfying
\[ \psi(0,r)=f(r),\quad \psi_t(0,r)=g(r),\quad r \in [0,\tfrac32] \]
and a blowup time $T\in (\frac12,\frac32)$ such that
\begin{equation}
\label{eq:convergence}
(T-t)^\frac32 \|(\psi(t,\cdot),\psi_t(t,\cdot))-(\psi^T(t,\cdot),\psi^T_t(t,\cdot))\|_{\mc E(T-t)}
\leq C_\varepsilon (T-t)^{|\omega_0|-\varepsilon} 
\end{equation}
for all $t \in [0,T)$ with $\omega_0:=\max\{-\frac32,\mu_0\}$ where $\mu_0$ 
is the real part of the first stable eigenvalue of $\psi^T$
and $C_\varepsilon>0$ is a constant which depends on $\varepsilon$.
\end{theorem}

Remarks:

\begin{itemize}
\item As usual, by a ``solution'' we mean a function which satisfies the equation in the sense of 
Duhamel.

\item By a simple scaling argument one immediately sees that 
\[ \|(\psi^T(t,\cdot),\psi^T_t(t,\cdot))\|_{\mc E(T-t)}\simeq (T-t)^{-\frac32} \]
for all $t \in [0,T)$.
This explains the normalization in \eqref{eq:convergence} and shows that $\psi$
converges to the self--similar solution $\psi^T$ in the backward lightcone of the blowup point.

\item The rate of convergence in \eqref{eq:convergence} is dictated by the first stable eigenvalue
which complies with heuristic expectations and numerics \cite{BC05}.
The $\varepsilon$--loss in our estimate is purely technical.
We also remark that the numerically obtained value for $\mu_0$ is approximately
$-0.59$ \cite{BC05}.

\item The $\mc E$--norm is very natural since it is derived from a conserved quantity of a suitable
``free'' equation which is associated to Eq.~\eqref{eq:main}, see below.
Furthermore, the boundary conditions assumed for the initial data $(f,g)$ are natural too, 
since any sufficiently regular solution $\psi$ of Eq.~\eqref{eq:main} must satisfy
$\psi(t,0)=\psi_t(t,0)=\psi_r(t,0)=0$ for all $t$ (provided $\psi$ belongs to the same topological
sector as $\psi^T$ which we obviously assume).

\item It is clear that the result cannot hold in the energy topology.
This is due to the fact that the corresponding local energy in the backward lightcone of the
blowup point \emph{decays} like $T-t$ as $t \to T-$ and thus, self--similar blowup is 
invisible in the energy topology.
This is, of course, nothing but a manifestation of energy supercriticality.
\end{itemize}

\subsection{An outline of the proof}
The proof consists of a perturbative construction around $\psi^T$ which proceeds in several steps.

\begin{enumerate}
\item First, we identify a suitable Hilbert space where the corresponding inner product is 
\begin{itemize}
\item derived from a conserved quantity of a suitable ``free'' equation where the latter is
(roughly speaking) obtained
from Eq.~\eqref{eq:main} by dropping the nonlinear term,
\item strong enough to detect self--similar blowup,
\item strong enough to control the nonlinearity.
\end{itemize}
\item Next, we introduce similarity coordinates, linearize Eq.~\eqref{eq:main} around $\psi^T$ and
construct a semigroup that governs the linearized evolution.
The application of semigroup theory to this problem is natural since 
\begin{itemize}
\item the involved differential operator is highly nonself--adjoint due to the introduction of
nonorthogonal coordinates,
\item the evolution problem is restricted to the backward lightcone of the blowup point and is thus
only well--posed in forward time which is reminiscent of parabolic equations.
\end{itemize}
\item Then we perform a detailed spectral analysis of the semigroup generator and 
construct a Riesz projection of rank $1$ which
removes the unstable symmetry mode that results from the time translation invariance of Eq.~\eqref{eq:main}.
As a consequence, we obtain \emph{exponential} decay of the linear evolution on the codimension $1$ 
stable subspace.
\item Next, we prove a Lipschitz property of the nonlinearity
which allows us to run a fixed point argument in order to obtain a solution to the nonlinear problem.
That this is possible is not surprising since our norm controls sufficiently many derivatives
to obtain a Moser--type estimate.
Furthermore, the linear evolution decays exponentially and this kind of decay is reproduced by the 
Duhamel formula.
However, in order to suppress the instability of the linear evolution we have to modify the data.
This is similar to the Lyapunov--Perron method in dynamical systems theory.
\item Finally, we show that the aforementioned modification of the data is equivalent to changing
the blowup time.
Thus, by choosing the appropriate blowup time we obtain a solution of the original equation
\eqref{eq:main} with the properties stated in Theorem \ref{thm:main}.
\end{enumerate}

\subsection{Notation}
As usual, we write $a\lesssim b$ if $a\leq cb$ for some $c>0$.
Similarly, we use $\gtrsim$ and $a\simeq b$ means $a\lesssim b$ and $b\lesssim a$.
The big--O symbol has its standard meaning from asymptotic analysis.
In order to improve readability we use boldface letters for vectors and number the individual components
by subscripts, e.g.~$\mb u=(u_1,u_2)$.
The symbol $D^F$ denotes the Fr\'echet derivative and $D^F_j$ is used for the $j$--th partial
Fr\'echet derivative.
For a closed linear operator $A$ we write $\sigma(A)$, $\sigma_p(A)$, $R_A(\lambda)$ for the
spectrum, point spectrum and resolvent of $A$, respectively.

\section{Transformation to a first--order system and similarity coordinates}

\subsection{Formulation of the Cauchy problem}
As explained in the introduction, we intend to study the Cauchy problem \footnote{In order to
avoid notational clutter we usually omit the arguments and write $\psi$ instead of $\psi(t,r)$.}
\begin{equation}
\label{eq:mainCauchy}
\left \{ \begin{array}{l}
\psi_{tt}-\psi_{rr}-\tfrac{2}{r}\psi_r+\tfrac{3}{r^2}F(\psi)=0 \\
\psi(0,\cdot)=f, \quad \psi_t(0,\cdot)=g
\end{array} \right .
\end{equation}
for a function
$\psi: \mc{C}_T \to \mathbb{R}$
where 
$$ \mc{C}_T:=\{(t,r): t \in [0,T), r \in [0,T-t]\},\quad T>0 $$
and $f,g: [0,T] \to \mathbb{R}$ are prescribed, sufficiently regular functions (the initial data).
Furthermore, the nonlinearity $F$ is given by $F(\psi)=\psi(\psi+1)(\psi+2)$.
Note also that the requirement of regularity at the center demands
$\psi_r(t,0)=0$
and $\psi(t,0) \in \{0,-1,-2\}$ for all $t$.
We are interested in the stability of the blowup solution
$$ \psi^T(t,r)=W_0(\tfrac{r}{T-t})-1 $$
where 
$$ W_0(\rho)=\frac{1-\rho^2}{1+\frac35 \rho^2} $$
is the Bizo\'n solution.
Since $\psi^T(t,0)=0$ we restrict ourselves to solutions of Eq.~\eqref{eq:mainCauchy} that satisfy
$\psi(t,0)=0$ for all $t$.
In the following we perform some formal manipulations to transform \eqref{eq:mainCauchy} into
a convenient form suitable for further analysis.

We intend to study small perturbations of $\psi^T$ and thus,
it is reasonable to reformulate \eqref{eq:mainCauchy} relative to $\psi^T$, i.e., we 
insert the ansatz $\psi=\psi^T+\varphi$ into Eq.~\eqref{eq:mainCauchy} and obtain the Cauchy problem
\begin{equation}
\label{eq:mainCauchyrel}
\left \{ \begin{array}{l}
\varphi_{tt}-\varphi_{rr}-\tfrac{2}{r}\varphi_r+\tfrac{6}{r^2}\varphi+\tfrac{3}{r^2}[F'(\psi^T)-2]\varphi
+\tfrac{1}{r^2}N_T(\varphi)=0 \\
\varphi(0,\cdot)=f-\psi^T(0,\cdot), \quad \varphi_t(0,\cdot)=g-\psi^T_t(0,\cdot)
\end{array} \right .
\end{equation}
for the perturbation $\varphi: \mc{C}_T \to \mathbb{R}$.
Here, 
\begin{align} 
\label{eq:nl}
N_T(\varphi)&=3[F(\psi^T+\varphi)-F(\psi^T)-F'(\psi^T)\varphi] \\
&=9(\psi^T+1)\varphi^2+3\varphi^3 \nonumber
\end{align}
is the nonlinear remainder.
Observe further that  
$$ F'(\psi^T(t,0))=F'(0)=2 $$
and thus, by subtracting the constant $2$ we have regularized the ``potential term'' 
in such a way that $\frac{F'(\psi^T(t,r))-2}{r^2}$ remains bounded as $r \to 0+$.
Finally, the perturbation $\varphi$ inherits the boundary conditions $\varphi(t,0)=\varphi_r(t,0)=0$
for all $t$.
So far nothing has happened and Eq.~\eqref{eq:mainCauchyrel} is equivalent to 
Eq.~\eqref{eq:mainCauchy} if $\psi=\psi^T+\varphi$.
In order to fix terminology we call
$$ \varphi_{tt}-\varphi_{rr}-\tfrac{2}{r}\varphi_r+\tfrac{6}{r^2}\varphi=0 $$
the \emph{free equation},
$$ \varphi_{tt}-\varphi_{rr}-\tfrac{2}{r}\varphi_r+\tfrac{6}{r^2}\varphi
+\tfrac{3}{r^2}[F'(\psi^T)-2]\varphi=0 $$
the \emph{linear} or \emph{linearized equation} and, finally, the full problem
Eq.~\eqref{eq:mainCauchyrel} is referred to as the \emph{nonlinear equation}.
Note carefully that we have assigned all singular terms to the free equation.
This is necessary since our overall strategy is to treat the nonlinear equation as a perturbation
of the linearized equation which, in turn, is viewed as a perturbation of the free equation.
Therefore, the topology is dictated by the free equation.

\subsection{Higher energy norm}
\label{subsec:higherenergy}
As already outlined in the introduction, the energy topology is too weak to study self--similar 
blowup.
Consequently, we have to find a stronger norm and it is advantageous if  
this norm is naturally associated to the free equation.
Furthermore, we intend to control the nonlinearity by a Moser--type estimate and therefore, we 
expect to need at least $\frac{5}{2}+$ derivatives (recall that the Yang--Mills problem
is in $1+5$ dimensions).
For simplicity, however, we avoid fractional Sobolev spaces and aim
for a norm that controls $3$ derivatives.
The key observation in this respect is that, if we set 
$$ \hat{\varphi}(t,r):=\tfrac{1}{r}\partial_r[\tfrac{1}{r}\partial_r(r^3 \varphi(t,r))], $$
we obtain the identity
$$ \hat{\varphi}_{tt}-\hat{\varphi}_{rr}=
\tfrac{1}{r}\partial_r \left \{\tfrac{1}{r}\partial_r\left [
r^3 \left (\varphi_{tt}-\varphi_{rr}-\tfrac{2}{r}\varphi_r
+\tfrac{6}{r^2}\varphi \right ) \right ]\right \}. $$
Thus, if $\varphi$ satisfies the free equation then $\hat{\varphi}$ is a solution to the 
one--dimensional wave equation on the half--line.
Furthermore, since $\varphi(t,0)=\varphi_r(t,0)=0$ implies 
$\hat{\varphi}(t,0)=0$, it follows that
\begin{equation} 
\label{eq:higherenergy}
\int_0^\infty \left [ \hat{\varphi}_t(t,r)^2+\hat{\varphi}_r(t,r)^2\right ]dr=\mathrm{const}. 
\end{equation}
The point is that the conserved quantity \eqref{eq:higherenergy} induces a stronger topology
than the energy since it contains third derivatives of $\varphi$.
Consequently, we refer to \eqref{eq:higherenergy} as a \emph{higher energy} for the free equation.
Moreover, if we truncate the domain of integration in \eqref{eq:higherenergy} to the backward
lightcone $\mc{C}_T$, we obtain a local version of the higher energy given by
\begin{equation}
\label{eq:lochigherenergy}
\int_0^{T-t}
\left [ \hat{\varphi}_t(t,r)^2+\hat{\varphi}_r(t,r)^2 \right ]dr. 
\end{equation}
A simple scaling argument (or a straightforward computation) then shows that the higher energy for the blowup 
solution $\psi^T$ behaves like $(T-t)^{-3}$
and thus, unlike the original energy, the local higher energy \eqref{eq:lochigherenergy} is strong enough to detect 
self--similar blowup.
Consequently, we study the Cauchy problem Eq.~\eqref{eq:mainCauchyrel} in the topology induced
by \eqref{eq:lochigherenergy}.

\subsection{First--order formulation}
We intend to formulate Eq.~\eqref{eq:mainCauchyrel} as a first--order system in time.
To this end, we introduce two auxiliary fields $\varphi_1$, $\varphi_2$ by
\begin{align}
\label{eq:varphi12}
\varphi_1(t,r)&:=\tfrac{r^3}{(T-t)^2} \varphi_t(t,r) \\
\varphi_2(t,r)&:=(T-t)\tfrac{1}{r}\partial_r \left [\tfrac{1}{r}\partial_r \left (r^3 \varphi(t,r) 
\nonumber
\right ) \right ]. 
\end{align}
The definition of the field $\varphi_2$ is motivated by the discussion in Section 
\ref{subsec:higherenergy}.
In fact, apart from the factor $T-t$ in front, $\varphi_2$ is exactly the function 
$\hat{\varphi}$ from Section \ref{subsec:higherenergy}.
The factor $T-t$ is introduced to put $\varphi_2$ on the same scaling level as the original
field $\varphi$.
The field $\varphi_1$ is a suitably scaled time derivative of $\varphi$ which leads to a simple
expression for the higher energy \eqref{eq:lochigherenergy} in terms of $\varphi_1$ and $\varphi_2$.
Note further that, on any time slice $t=\mathrm{const}$, $\varphi$ can be reconstructed from 
$\varphi_2$ by 
$$ \varphi(t,r)=\tfrac{1}{(T-t)r^3}(K^2 \varphi_2(t,\cdot))(r) $$
where the integral operator
$$ Kf(\rho):=\int_0^\rho s f(s)ds $$
will appear frequently in the sequel.
Furthermore, a straightforward computation shows
$$ \varphi_{rr}+\tfrac{2}{r}\varphi_r-\tfrac{6}{r^2}\varphi
=\tfrac{1}{T-t}\left [\tfrac{1}{r} \varphi_2-\tfrac{3}{r^3}K\varphi_2 \right ] $$
where $K\varphi_2$ is an abbreviation for $(K\varphi_2(t,\cdot))(r)$.
Consequently, Eq.~\eqref{eq:mainCauchyrel} transforms into
\begin{equation}
\label{eq:varphi1st}
\left \{\begin{array}{ll}
\partial_t \varphi_1=\frac{2\varphi_1}{T-t}+\frac{r^2 \varphi_2}{(T-t)^3}
-\frac{3K\varphi_2}{(T-t)^3}-\frac{3F'(\psi^T)-6}{(T-t)^3r^2}K^2 \varphi_2
-\frac{r}{(T-t)^2}N_T \left (\frac{K^2\varphi_2}{(T-t)r^3} 
\right )\\
\partial_t \varphi_2=(T-t)^3\tfrac{1}{r}\partial_r \left (\tfrac{1}{r}\partial_r \varphi_1 \right )
-\tfrac{\varphi_2}{T-t}
\end{array} \right .
\end{equation}
for $\varphi_j: \mc{C}_T \to \mathbb{R}$, $j=1,2$,
with initial data
\begin{equation}
\label{eq:varphi1stdata1}
\begin{aligned}
\varphi_1(0,r)&=\tfrac{r^3}{T^2} [g(r)-\psi_t^T(0,r)] \\
\varphi_2(0,r)&=T(\tfrac{1}{r}\partial_r)^2 [r^3 (f(r)-\psi^T(0,r))].
\end{aligned}
\end{equation}
In order to write this in a more concise form we introduce a differential operator $\mc D^2$ given by
\begin{equation}
\label{eq:defD2}
\mc D^2 f(r):=r f''(r)+5 f'(r)+\tfrac{3}{r}f(r).
\end{equation}
Then we have $\mc D^2 f(r)=(\frac{1}{r}\partial_r)^2 [r^3 f(r)]$ and thus, the initial data
in Eq.~\eqref{eq:varphi1stdata1} can be written as
\begin{equation}
\label{eq:varphi1stdata}
\begin{aligned}
\varphi_1(0,r)&=\tfrac{r^3}{T^2} [g(r)-\psi_t^T(0,r)] \\
\varphi_2(0,r)&=T\mc D^2 [f-\psi^T(0,\cdot)](r).
\end{aligned}
\end{equation}

\subsection{Similarity coordinates}

Note that, via $\psi^T$, both the ``potential term'' and the nonlinearity in Eq.~
\eqref{eq:varphi1st} depend explicitly on $t$.
More precisely, they depend on the ratio $\frac{r}{T-t}$.
In view of the self--similar character of the problem it is thus natural to introduce
adapted coordinates (``similarity variables'') by setting
$$\tau:=-\log(T-t),\quad \rho:=\frac{r}{T-t}.$$ 
The inverse map is given by
$$ t=T-e^{-\tau},\quad r=e^{-\tau}\rho $$
and the derivatives transform according to
$$ \partial_t=e^\tau (\partial_\tau+\rho \partial_\rho),\quad \partial_r=e^\tau \partial_\rho. $$
Furthermore, under the transformation $(t,r) \mapsto (\tau,\rho)$, the backward lightcone $\mc{C}_T$
is mapped to the infinite cylinder 
$$ \mc{Z}_T:=\{(\tau,\rho): \tau\geq -\log T, \rho \in [0,1]\}. $$
Consequently, in the new coordinates $(\tau,\rho)$ the blowup takes place at infinity.
By setting 
$$\phi_j(\tau,\rho):=\varphi_j(T-e^{-\tau},e^{-\tau} \rho),\quad j=1,2 $$ we obtain from 
Eq.~\eqref{eq:varphi1st} the system
\begin{equation}
\label{eq:phi1st}
\left \{
\begin{array}{l}
\partial_\tau \phi_1=-\rho \partial_\rho \phi_1+2\phi_1+\rho^2 \phi_2-3 K\phi_2
-V(\rho)K^2\phi_2-\rho N_T\left (\frac{1}{\rho^3}K^2\phi_2 \right )\\
\partial_\tau \phi_2=\tfrac{1}{\rho}\partial_\rho\left (\frac{1}{\rho}\partial_\rho \phi_1 \right)
-\rho \partial_\rho \phi_2-\phi_2
\end{array} \right .
\end{equation}
for functions $\phi_j: \mc{Z}_T \to \mathbb{R}$, $j=1,2$, with data
\begin{equation}
\label{eq:phi1stdata}
\begin{aligned}
\phi_1(-\log T,\rho)&=T\rho^3[g(T\rho)-\psi_t^T(0,T\rho)] \\
\phi_2(-\log T,\rho)&=T\mc D^2[f-\psi^T(0,\cdot)](T\rho)
\end{aligned}
\end{equation}
and the potential
\begin{equation}
\label{eq:V}
V(\rho)=\frac{3F'(W_0(\rho)-1)-6}{\rho^2}=-144 \frac{5-\rho^2}{(5+3\rho^2)^2}.
\end{equation}
Note carefully that, by transforming to similarity variables, all the explicit dependencies 
on the time variable have disappeared and we 
have effectively
reduced the study of the self--similar blowup solution $\psi^T$ to a small data asymptotic stability
problem given by Eq.~\eqref{eq:phi1st}.
The analysis of Eq.~\eqref{eq:phi1st} is the content of the present paper.

\section{Linear perturbation theory}
\label{sec:lin}

In this section we study the linearized problem that results from Eq.~\eqref{eq:phi1st} by
dropping the nonlinear term.
Actually, we start with the free problem which follows from Eq.~\eqref{eq:phi1st} by dropping the
nonlinearity \emph{and} the potential term.
Our approach is operator--theoretic.
The point is that we need to employ semigroup theory in order to solve the free problem
since the transformation to the nonorthogonal coordinate system $(\tau,\rho)$ has in fact destroyed
the underlying self--adjoint structure of the wave operator.
Consequently, we rewrite Eq.~\eqref{eq:phi1st} as an ordinary differential equation (in $\tau$)
on a suitable Hilbert space which is dictated by the local higher energy defined in 
\eqref{eq:lochigherenergy}.
Then we prove well--posedness of the free problem by an application of the Lumer--Phillips theorem.
The analogous result for the linearized problem follows by a general abstract perturbation argument, 
although the corresponding growth bound of the evolution that is obtained by this procedure 
is far from being optimal.
In order to improve this bound, we perform a more detailed spectral analysis.
It turns out that the linearized time evolution exhibits an inherent instability which is a 
manifestation of the time translation invariance of the original problem.
We show how to construct a suitable spectral projection that removes this ``artificial'' 
instability and proceed by proving a decay bound for the linearized evolution on the stable subspace.
This result, which is almost optimal, concludes the study of the linearized problem.

\subsection{Function spaces and well--posedness of the linear problem}
We are going to need the following version of Hardy's inequality.

\begin{lemma}
\label{lem:Hardy}
Let $\alpha>1$ and assume that $u \in C[0,1]$ has a weak derivative as well as
$$ \lim_{\rho \to 0+}\frac{|u(\rho)|^2}{\rho^{\alpha-1}}=0. $$
Then
$$ \int_0^1 \frac{|u(\rho)|^2}{\rho^\alpha}d\rho\leq \left (\frac{2}{\alpha-1}\right )^2
\int_0^1 \frac{|u'(\rho)|^2}{\rho^{\alpha-2}}d\rho. $$
\end{lemma}

\begin{proof}
This follows by integration by parts and the Cauchy--Schwarz inequality.
\end{proof}

We set 
$$ \tilde{\mc{H}}:=\left \{\mb{u}=(u_1,u_2)\in C^4[0,1] \times C^1[0,1]: u_1^{(k)}(0)=u_2(0)=0,\: 
k=0,1,2,3 \right \} $$
and define a sesquilinear form $(\cdot|\cdot)$ by
$$ (\mb{u}|\mb{v}):=(u_1|v_1)_1+(u_2|v_2)_2:=
\int_0^1 D^2 u_1(\rho)\overline{D^2 v_1(\rho)}d\rho
+\int_0^1 u_2'(\rho)\overline{v_2'(\rho)}d\rho $$
where $Df(\rho):=\frac{1}{\rho}f'(\rho)$.
Note that $(\cdot|\cdot)$ is chosen in such a way that it leads to the local higher energy
Eq.~\eqref{eq:lochigherenergy}.

\begin{lemma}
\label{lem:basicH}
The sesquilinear form $(\cdot|\cdot)$ defines an inner product on $\tilde{\mc{H}}$ and the
completion of $\tilde{\mc{H}}$, denoted by $\mc{H}$, is a Hilbert space.
Furthermore, the subspace $C^\infty_c(0,1] \times C^\infty_c(0,1]$ of $\mc{H}$ is dense
and $\mb{u} \in \mc{H}$ implies $\mb{u} \in C^1[0,1]\times C[0,1]$ with the boundary conditions
$u_1(0)=u_1'(0)=u_2(0)=0$.
\end{lemma}

\begin{proof}
From Hardy's inequality we obtain the estimate
$$ \int_0^1 |D^2 u_1(\rho)|^2d\rho \lesssim \int_0^1 |u_1^{(4)}(\rho)|^2 d\rho $$
for $\mb{u}\in \tilde{\mc{H}}$
which shows that $(\cdot|\cdot)$ is well--defined on all of $\tilde{\mc{H}}\times \tilde{\mc{H}}$.
Furthermore, the assumed boundary conditions ensure that $(\mb{u}|\mb{u})=0$ if and only if
$\mb{u}=0$.
By the density of $C^\infty_c(0,1]$ in $L^2(0,1)$ we can, for any $\epsilon>0$, find a function
$\tilde{v} \in C^\infty_c(0,1]$ such that $\|D^2 u_1-\tilde{v}\|_{L^2(0,1)}<\epsilon$.
By setting $v:=K^2 \tilde{v}$ we obtain $v\in C^\infty_c(0,1]$ with $\|D^2(u_1-v)\|_{L^2(0,1)}<\epsilon$
which yields the claimed density property.
Finally, we note that
\[ |\tfrac{1}{\rho}u_1'(\rho)|\leq \int_0^1 |\partial_\rho[\tfrac{1}{\rho}u_1'(\rho)]|d\rho
\leq \int_0^1 |D^2 u_1(\rho)|^2 d\rho \]
by Cauchy--Schwarz.
\end{proof}

Now we set 
$$ \mc{D}(\tilde{\mb{L}}_0):=\left \{\mb{u}=(u_1,u_2)\in C^\infty[0,1] \times C^\infty[0,1]: 
u_1^{(k)}(0)=u_2(0)=0,\:\: k=0,1,2,3,4 \right \}$$ and define a differential operator
on $\mc{D}(\tilde{\mb{L}}_0)$ by 
$$ \tilde{\mb{L}}_0\mb{u}(\rho):=\left ( \begin{array}{c}
-\rho u_1'(\rho)+2u_1(\rho)+\rho^2 u_2(\rho)-3Ku_2(\rho) \\
D^2 u_1(\rho)-\rho u_2'(\rho)-u_2(\rho) \end{array} \right ) $$
where, as before, $Kf(\rho):=\int_0^\rho sf(s)ds$.
At this point it is important to note that
$$ \rho^2 u_2(\rho)-3 K u_2(\rho)=O(\rho^4) $$
instead of only $O(\rho^3)$ as one might expect at first glance.
This is due to a special cancellation.
As a consequence we see that 
$$[\tilde{\mb{L}}_0\mb{u}]_1^{(k)}(\rho)=O(\rho^{4-k}),\quad k=0,1,2,3,4 $$ 
where $[\tilde{\mb{L}}_0 \mb{u}]_j$, $j=1,2$, denotes the $j$--th component of 
$\tilde{\mb{L}}_0\mb{u}$.
Similarly, we have 
$$[\tilde{\mb{L}}_0 \mb{u}]_2^{(k)}(\rho)=O(\rho^{1-k})$$ 
for $k=0,1$ and we conclude
that $\tilde{\mb{L}}_0$ has range in $\tilde{\mc{H}}$.
Comparison with Eq.~\eqref{eq:phi1st} shows that $\tilde{\mb{L}}_0$ represents the right--hand side
of the free problem and Lemma \ref{lem:basicH} implies that $\tilde{\mb{L}}_0$ is densely defined.
Furthermore, in view of the definitions of $\varphi_1, \varphi_2$ in \eqref{eq:varphi12}, the
boundary conditions required in $\mc{D}(\tilde{\mb{L}}_0)$ are natural.

\begin{lemma}
\label{lem:WPfree}
The operator $\tilde{\mb{L}}_0: \mc{D}(\tilde{\mb{L}}_0)\subset \mc{H} \to \mc{H}$ is 
closable and its closure $\mb{L}_0$ generates a strongly continuous one--parameter semigroup 
$\mb{S}_0: [0,\infty) \to \mc{B}(\mc{H})$ that satisfies
$$ \|\mb{S}_0(\tau)\|\leq e^{-\frac32 \tau} $$
for all $\tau \geq 0$.
In particular, the Cauchy problem
$$ \left \{
\begin{array}{l}
\frac{d}{d\tau}\Phi(\tau)=\mb{L}_0 \Phi(\tau) \\
\Phi(0)=\mb{u} \in \mc{H}
\end{array} \right . $$
has a unique mild solution $\Phi: [0,\infty) \to \mc{H}$ 
given by $\Phi(\tau)=\mb{S}_0(\tau)\mb{u}$.
\end{lemma}

\begin{proof}
According to the Lumer--Phillips Theorem (see \cite{EN00}, p.~~83, Theorem 3.15), it suffices to
show that 
\begin{itemize}
\item $\mathrm{Re}(\tilde{\mb{L}}_0\mb{u}|\mb{u})\leq -\frac32 \|\mb{u}\|^2$ for all $\mb{u} \in 
\mc{D}(\tilde{\mb{L}}_0)$ and
\item the range of $\lambda-\tilde{\mb{L}}_0$ is dense in $\mc{H}$ for some $\lambda>-\frac32$.
\end{itemize}
In the following, we employ a common
abuse of notation and use the symbol $\rho$ to denote both the independent variable and the
identity function.
Furthermore, all integrals run from $0$ to $1$ and we omit denoting the measure $d\rho$.
In order to estimate $\mathrm{Re}(\tilde{\mb{L}}_0 \mb{u}|\mb{u})$, we start by collecting all terms
that only contain $u_1$ and integrate by parts to obtain
\begin{align*} 
-\mathrm{Re}\int D^2 (\rho^2 D u_1)\overline{D^2 u_1}+2\int |D^2 u_1|^2
&=-\Re \int \rho (D^2 u_1)' \overline{D^2 u_1}-2 \int |D^2 u_1|^2 \\
&=-\tfrac12 |D^2 u_1(1)|^2-\tfrac32 \int |D^2 u_1|^2
\end{align*}
where we have used the commutator $[D^2, \rho^2 D]=4D^2$ and $\rho u_1'(\rho)=\rho^2 Du_1(\rho)$.
The boundary term at $0$ vanishes thanks to $u_1^{(k)}(0)=0$ for $k=0,1,2,3,4$.
Similarly, the terms containing only $u_2$ are given by
\begin{align*}
-\Re \int (\rho u_2')'\overline{u_2'}-\int |u_2'|^2&=
-\Re \int \rho u_2'' \overline{u_2'}-2\int |u_2'|^2\\
&=-\tfrac12 |u_2'(1)|^2-\tfrac32 \int |u_2'|^2.
\end{align*}
As a consequence, it suffices to show that the mixed terms are dominated by 
$\frac{1}{2}(|D^2 u_1(1)|^2+|u_2'(1)|^2)$ and
indeed we have
\begin{align*}
\Re &\int [D^2(\rho ^2 u_2)-3Du_2]\overline{D^2 u_1}+\Re \int (D^2 u_1)'\overline{u_2'} \\
&=\Re \int u_2'' \overline{D^2 u_1}+\Re [D^2 u_1(1) \overline{u_2'(1)}]-\Re \int (D^2 u_1)
\overline{u_2''} \\
&=\Re [D^2 u_1(1) \overline{u_2'(1)}]\leq \tfrac{1}{2}(|D^2 u_1(1)|^2+|u_2'(1)|^2)
\end{align*}
since $D^2 u_1(0)=0$.
Thus, we obtain $\Re(\tilde{\mb{L}}_0 \mb{u}|\mb{u})\leq -\frac32 \|\mb{u}\|^2$ as desired.
Note that this result is not surprising since it is just a reflection
of the fact that the higher energy \eqref{eq:higherenergy} is conserved for the free problem and
the factor $-\frac32$ can be concluded by a scaling argument.

It remains to show that the range of $\lambda-\tilde{\mb{L}}_0$ is dense in $\mc{H}$ for some 
$\lambda>-\frac32$.
To this end it suffices to show that the equation $(2-\tilde{\mb{L}}_0)\mb{u}=\mb{f}$ has a solution
$\mb{u}\in \mc{D}(\tilde{\mb{L}}_0)$ for any $\mb{f} \in C^\infty_c(0,1] \times C^\infty_c(0,1]$
(cf.~Lemma \ref{lem:basicH}).
The point is that the equation $(2-\tilde{\mb{L}}_0)\mb{u}=\mb{f}$ can be solved explicitly by 
elementary ODE methods.
We just state the result.
For given $\mb{f}=(f_1,f_2) \in C^\infty_c(0,1]\times C^\infty_c(0,1]$ 
define an auxiliary function $u$ by 
$$ u(\rho):=\frac{\rho^3}{(1-\rho^2)^2}\int_\rho^1 \frac{1-s^2}{s^4}\left [f_1(s)+s^2 Kf_2(s)
\right ]ds. $$
Observe that $u \in C^\infty[0,1]$ by Taylor expansion and $u^{(k)}(\rho)=O(\rho^{3-k})$ for
$k=0,1,2,3$.
Now set $u_2:=Du$.
Then we have $u_2 \in C^\infty[0,1]$ and $u_2(\rho)=O(\rho)$.
Furthermore, define $u_1:=K(\rho^2 Du)+Ku-K^2 f_2$ which implies
$u_1 \in C^\infty[0,1]$ and $u_1^{(k)}(\rho)=O(\rho^{5-k})$ for $k=0,1,\dots,5$.
Consequently, we obtain $\mb{u}=(u_1,u_2) \in \mc{D}(\tilde{\mb{L}}_0)$ and by straightforward
differentiation one verifies that indeed $(2-\tilde{\mb{L}}_0)\mb{u}=\mb{f}$.
Since $\mb{f}$ was arbitrary we are done.
\end{proof}

Next, we add the potential term from Eq.~\eqref{eq:phi1st} which is represented by the operator
$\mb{L}'$, defined
by
$$ \mb{L}' \mb{u}(\rho):=\left ( 
\begin{array}{c}
-V K^2 u_2 \\
0
\end{array} \right ) $$
with the smooth potential $V$ given explicitly in Eq.~\eqref{eq:V}. 
Since $V \in C^\infty[0,1]$, it follows that $\mb{L}' \in \mc{B}(\mc{H})$ (use Hardy's inequality) 
and we can immediately conclude
the well--posedness of the linearized problem.

\begin{corollary}
\label{cor:WPlin}
The operator $\mb{L}:=\mb{L}_0+\mb{L}'$ 
generates a strongly continuous one--parameter semigroup $\mb{S}: [0,\infty)
\to \mc{H}$ which satisfies
$$ \|\mb{S}(\tau)\|\leq e^{(-\frac32+\|\mb{L}'\|)\tau} $$
for all $\tau \geq 0$. In particular, the Cauchy problem
$$ \left \{ \begin{array}{l}
\frac{d}{d\tau}\Phi(\tau)=\mb{L} \Phi(\tau) \\
\Phi(0)=\mb{u} \in \mc{H}
\end{array} \right . $$
has a unique mild solution $\Phi: [0,\infty) \to \mc{H}$ 
given by $\Phi(\tau)=\mb{S}(\tau)\mb{u}$.
\end{corollary}

\begin{proof}
This is a consequence of the Bounded Perturbation Theorem, see \cite{EN00}, p.~158.
\end{proof}

\subsection{Spectral analysis of the generator}

In order to improve the rough growth bound for the linearized evolution given in Corollary 
\ref{cor:WPlin}, we have to analyze the spectrum of $\mb{L}$.
Note first that the growth bound for the free evolution in Lemma \ref{lem:WPfree} implies
\begin{equation}
\label{eq:specL0}
\sigma(\mb{L}_0)\subset \{\lambda \in \mathbb{C}: \mathrm{Re}\lambda \leq -\tfrac32 \}, 
\end{equation}
see \cite{EN00}, p.~55, Theorem 1.10.
In fact, it is not very hard to see that we have equality here, i.e., the growth bound in
Lemma \ref{lem:WPfree} is sharp.
However, we will not need this result in the following and therefore we omit its proof.
Of course, the addition of the potential term $\mb{L}'$ changes the spectrum; but as a consequence of the 
following result, the change is in some sense the mildest possible: it only affects the point
spectrum.

\begin{lemma}
\label{lem:compact}
The operator $\mb{L}': \mc{H} \to \mc{H}$ is compact.
As a consequence, $\sigma(\mb{L})\backslash \sigma(\mb{L}_0) \subset \sigma_p(\mb{L})$.
\end{lemma}

\begin{proof}
We write $\mc{H}=\mc{H}_1 \times \mc{H}_2$ and denote by $\|\cdot\|_j$, $j=1,2$, the respective norms on 
$\mc{H}_j$.
Since multiplication by $V$ is bounded as an operator from $\mc{H}_1$ to $\mc{H}_1$ (Hardy's inequality), 
it suffices to show that $K^2$ is compact as an operator from $\mc{H}_2$ to $\mc{H}_1$.
Let $(u_j) \subset \mc{H}_2$ be a bounded sequence.
By definition of $\|\cdot\|_2$ and the boundary condition $u_j(0)=0$ 
it follows that $(u_j) \subset H^1(0,1)$ is bounded and the compact embedding
$H^1(0,1) \subset \subset L^2(0,1)$ implies that $(u_j)$ has a subsequence which converges
in $L^2(0,1)$.
Since $\|K^2 u_j\|_1=\|u_j\|_{L^2(0,1)}$ we conclude that $(K^2 u_j)$ has a convergent
subsequence in $\mc{H}_1$ which implies the compactness of $K^2$.

If $\lambda \in \sigma(\mb{L})\backslash \sigma(\mb{L}_0)$ then it follows from the identity
$\lambda-\mb{L}=[1-\mb{L}'\mb{R}_{\mb{L}_0}(\lambda)](\lambda-\mb{L}_0)$ 
and the spectral theorem for compact operators
(Riesz--Schauder theory, see e.g.~ \cite{W00}, Section 5.4) that $\lambda \in \sigma_p(\mb{L})$.
\end{proof}

As we will show now, Lemma \ref{lem:compact} provides the link between the spectral problem for
$\mb{L}$ and the mode stability ODE \eqref{eq:modstab}.

\begin{lemma}
\label{lem:specmodstab}
If $\lambda \in \sigma(\mb{L})$ and $\mathrm{Re}\lambda>-\frac32$ 
then there exists a nontrivial $u \in C^\infty[0,1]$
such that
\begin{equation}
\label{eq:modstab2}
-(1-\rho^2)\left [u''+\tfrac{2}{\rho}u' \right ]+2\lambda \rho u'
+\lambda(\lambda+1)u+\frac{3F'(W_0(\rho)-1)}{\rho^2}u=0. 
\end{equation}
\end{lemma}

\begin{proof}
Let $\lambda \in \sigma(\mb{L})$ with $\mathrm{Re}\lambda>-\frac32$. According to Lemma \ref{lem:compact}
and Eq.~\eqref{eq:specL0} we have
$\lambda \in \sigma_p(\mb{L})$ and thus, there exists a nontrivial 
$\mb{u} \in \mc{D}(\mb{L}_0)\subset \mc{H}$ 
such that $(\lambda-\mb{L})\mb{u}=\mb{0}$.
Writing out the components we obtain the two equations
\begin{equation}
\label{eq:spec1st}
\left \{
\begin{array}{l}
\lambda u_1(\rho)+\rho u_1'(\rho)-2u_1(\rho)-\rho^2 u_2(\rho)+3 K u_2(\rho)+V(\rho)K^2 u_2(\rho)
=0 \\
\lambda u_2(\rho)-D^2 u_1(\rho)+\rho u_2'(\rho)+u_2(\rho)=0.
\end{array}
\right.
\end{equation}
The second equation implies
\begin{equation}
\label{eq:u1}
u_1(\rho)=\int_0^\rho s^3 u_2(s)ds+(\lambda-1)K^2 u_2(\rho) 
\end{equation}
which in particular shows that $u_2$ is nonzero.
In view of Eq.~\eqref{eq:varphi12} we set
\begin{equation}
\label{eq:u2}
u(\rho):=\tfrac{1}{\rho^3}K^2 u_2(\rho) 
\end{equation}
and note that $u \in C^2[0,1]$.
With this definition the expression for $u_1$ in Eq.~\eqref{eq:u1} simplifies to
\begin{equation}
\label{eq:u11}
u_1(\rho)=\rho^3 [\rho u'(\rho)+\lambda u(\rho)].
\end{equation}
Inserting Eqs.~\eqref{eq:u11}, \eqref{eq:u2} into the first equation of \eqref{eq:spec1st}
we infer
$$ -(1-\rho^2)\left [u''(\rho)+\tfrac{2}{\rho}u'(\rho) \right ]+2\lambda \rho u'(\rho)+\lambda(\lambda+1)
u(\rho)+\left [\tfrac{6}{\rho^2}+V(\rho)\right ]u(\rho)=0 $$
and by recalling the definition of $V$ in Eq.~\eqref{eq:V} we see that $u$ indeed 
satisfies Eq.~\eqref{eq:modstab2}.
Note that the coefficients in Eq.~\eqref{eq:modstab2} belong to $C^\infty(0,1)$ and 
furthermore, the coefficient of $u''$ does not vanish in $(0,1)$ which shows that the 
solution $u$ is in $C^\infty(0,1)$ by basic ODE theory.
The behavior at the endpoints follows by Frobenius' method:
at $\rho=0$ the Frobenius indices are $\{-3,2\}$ and therefore, $u \in C^2[0,1]$ already implies 
$u \in C^\infty[0,1)$
with $u(\rho)=O(\rho^2)$ as $\rho \to 0+$.
At $\rho=1$ we have the indices $\{0,1-\lambda\}$.
Note that $\mb{u} \in \mc{H}$ implies $u_2 \in H^1(0,1)$ and thus, $u \in H^3(\frac12,1)$.
Since $\mathrm{Re}(1-\lambda)<\frac52$ by assumption, the condition $u \in H^3(\frac12,1)$ 
excludes \footnote{Strictly speaking, the cases $\lambda \in \{-1,0,1\}$ require special attention since
for these values of $\lambda$ there exist two possibilities: the nonsmooth solution involves a 
logarithmic term or all solutions are smooth at $\rho=1$. In either case, however, we arrive at the same
conclusion as for $\lambda \notin \{-1,0,1\}$.} 
the nonsmooth solution at $\rho=1$ and we obtain $u \in C^\infty[0,1]$
as claimed. 
\end{proof}

\subsection{Construction of the spectral projection}

As already mentioned in the introduction, the function $g(\rho):=\rho W_0'(\rho)$ solves the mode 
stability ODE \eqref{eq:modstab2} with $\lambda=1$.
Via the transformations in the proof of Lemma \ref{lem:specmodstab} (in particular
Eqs.~\eqref{eq:u1} and \eqref{eq:u2}), 
$g$ gives rise to a function
$\mb{g} \in \mc{D}(\tilde{\mb{L}}_0)$ which (after a convenient normalization) reads explicitly
\begin{align}
\label{eq:g}
\mb{g}(\rho)&=-\tfrac{1}{240} \left (
\rho^3[ \rho g'(\rho)+g(\rho)], 
\tfrac{1}{\rho} \partial_\rho[\tfrac{1}{\rho}\partial_\rho(\rho^3 g(\rho))]
\right ) \\
&=
\left (\frac{\rho^5(5-\rho^2)}{(5+3\rho^2)^3}, \frac{\rho(125-50\rho^2-3 \rho^4)}
{(5+3 \rho^2)^4} \right ) \nonumber
\end{align}
and satisfies $(1-\mb{L})\mb{g}=\mb{0}$.
Consequently, $1 \in \sigma_p(\mb{L})$ but, as already indicated in the introduction, this instability
is induced by the time translation symmetry of the Yang--Mills equation \eqref{eq:main}.
Our aim is to construct a suitable spectral projection that removes this symmetry mode.
As a preparation for this we need the following observation.

\begin{lemma}
\label{lem:isol}
The eigenvalue $1 \in \sigma_p(\mb{L})$ is isolated in the spectrum of $\mb{L}$ and its algebraic 
multiplicity is finite.
\end{lemma}

\begin{proof}
According to Lemma \ref{lem:specmodstab} each $\lambda \in \sigma(\mb{L})$ with 
$\mathrm{Re}\lambda>-\frac32$ gives rise to a nontrivial function $u \in C^\infty[0,1]$ that
satisfies Eq.~\eqref{eq:modstab2}.
In fact, inspection of the proof of Lemma \ref{lem:specmodstab} shows that $u$ is even 
analytic. 
Consequently, all $\lambda \in \sigma(\mb{L})$ with $\mathrm{Re}\lambda > -\frac32$ are zeros
of an analytic function (namely the Wronskian of the two analytic solutions of Eq.~\eqref{eq:modstab2}
around $\rho=0$ and $\rho=1$, respectively) and therefore they are isolated.
If the algebraic multiplicity of $1 \in \sigma_p(\mb{L})$ were infinite then, by
\cite{K80}, p.~239, Theorem 5.28, $1$ would belong to the essential spectrum \footnote{
There exist at least five nonequivalent notions of essential spectra for nonself--adjoint operators,
see \cite{EE87} for a detailed discussion.
We stick to the definition given by Kato \cite{K80} as the set of all $\lambda$ such that
$\lambda-\mb{L}$ fails to be semi--Fredholm.} of $\mb{L}$.
However, since the essential spectrum is stable under compact perturbations (\cite{K80},
p.~244, Theorem 5.35) and $1 \notin \sigma(\mb{L}_0)$, 
we conclude that the algebraic multiplicity must be finite.
\end{proof}

Lemma \ref{lem:isol} allows us to define the Riesz projection
\begin{equation} 
\label{eq:P}
\mb{P}:=\tfrac{1}{2\pi i}\int_\Gamma (\lambda-\mb{L})^{-1}d\lambda 
\end{equation}
where $\Gamma$ is a circle that lies entirely in $\rho(\mb{L})$ and encloses the eigenvalue $1$ 
in such a way that no other spectral points of $\mb{L}$ lie inside $\Gamma$. 
By definition, the algebraic multiplicity of $1 \in \sigma_p(\mb{L})$ equals 
$\dim \rg \mb{P}$ and thus, by 
Lemma \ref{lem:isol}, $\mb{P}$
is of finite rank.
Moreover, $\mb{P}$ commutes with $\mb{L}$ in the sense that $\mb{PL} \subset \mb{LP}$ 
and as a consequence, 
$\mb{P}$ also commutes with the semigroup generated by $\mb{L}$, i.e., 
$\mb{PS}(\tau)=\mb{S}(\tau)\mb{P}$ for any $\tau \geq 0$.
We set $\mc{M}:=\rg \mb{P}$ which is a finite--dimensional subspace of $\mc{H}$ and denote by 
$\mb{L}_\mc{M}:=\mb{L}|_{\mc{D}(\mb{L})\cap \mc{M}}$ the part of $\mb{L}$ in $\mc{M}$.
$\mb{L}_\mc{M}$ is a linear bounded
operator on the finite--dimensional Hilbert space $\mc{M}$ with $\sigma(\mb{L}_\mc{M})=\{1\}$.
We refer to \cite{K80} for these standard facts.

\begin{lemma}
\label{lem:algmult}
The subspace $\mc{M}=\rg \mb{P}$ is one--dimensional and spanned by the symmetry mode $\mb{g}$.
\end{lemma}

\begin{proof}
Note first that it follows from the proof of Lemma \ref{lem:specmodstab} that the geometric 
eigenspace of $1 \in \sigma_p(\mb{L})$ is one--dimensional and spanned by $\mb{g}$. 
Consequently, since $1 \in \sigma(\mb{L}_\mc{M})=\sigma_p(\mb{L}_\mc{M})$, we conclude that
$\mb{g} \in \mc{M}$ which shows $\langle \mb{g} \rangle \subset \mc{M}$.

In order to prove the reverse implication observe that $\sigma(1-\mb{L}_\mc{M})=\{0\}$ and thus, 
$1-\mb{L}_\mc{M}$ is nilpotent.
This means that there exists an $m \in \mathbb{N}$ such that $(1-\mb{L}_\mc{M})^m \mb{u}=\mb{0}$ for all
$\mb{u} \in \mc{M}$ and we assume that $m$ is minimal with this property.
If $m=1$ it follows that $\mc{M} \subset \ker(1-\mb{L}_\mc{M})=\langle \mb{g} \rangle$ and we are done.
Thus, assume $m \geq 2$.
Then there exists a nonzero $\mb{v} \in \rg(1-\mb{L}_\mc{M})$ such that $(1-\mb{L}_\mc{M})\mb{v}=0$.
In other words, $\mb{g} \in \rg(1-\mb{L}_\mc{M})$, i.e., there exists a $\mb{u} \in \mc{D}(\mb{L})$ 
such that
$(1-\mb{L})\mb{u}=\mb{g}$.
By a similar computation as in the proof of Lemma \ref{lem:specmodstab} we infer the equation 
\begin{align}
\label{eq:algmult}
-&(1-\rho^2)\left [u''(\rho)+\tfrac{2}{\rho}u'(\rho) \right ]+2\rho u'(\rho)
+2u(\rho)+\frac{3F'(W_0(\rho)-1)}{\rho^2}u(\rho) \\
&=\tfrac{1}{\rho^3}g_1(\rho)
+\tfrac{1}{\rho}Kg_2(\rho)-\tfrac{1}{\rho^3}K^2 g_2(\rho)=
\frac{\rho^2(35-3\rho^2)}{3(5+3 \rho^2)^3}=:\tilde{g}(\rho) \nonumber
\end{align}
for the function
$u(\rho):=\tfrac{1}{\rho^3}K^2 u_2(\rho)$.
The homogeneous version of Eq.~\eqref{eq:algmult} has the fundamental system $\{h_0,h_1\}$ where
$$h_0(\rho)=\tfrac{\rho^2}{(5+3\rho^2)^2}$$
and $h_1$ (which can also be given in closed form) behaves like $|h_1(\rho)| \simeq \frac{1}{\rho^3}$
as $\rho \to 0+$ and $|h_1(\rho)| \simeq |\log(1-\rho)|$ as $\rho \to 1-$.
After a suitable normalization of $h_1$ we obtain for the Wronskian of $h_0$ and $h_1$ the expression
$$ W(h_0,h_1)(\rho)=\tfrac{1}{\rho^2(1-\rho^2)} $$
and thus, according to the variation of constants formula, $u$ must be of the form
$$ u(\rho)=c_0 h_0(\rho)+c_1 h_1(\rho)+h_0(\rho)\int_{\rho_0}^\rho s^2 h_1(s)\tilde{g}(s)ds
-h_1(\rho)\int_{\rho_1}^\rho s^2 h_0(s)\tilde{g}(s)ds $$
for suitable constants $c_0,c_1 \in \mathbb{C}$ and $\rho_0, \rho_1 \in [0,1]$.
Since $\lim_{\rho \to 0+}u(\rho)$ exists, we must have $c_1=\int_{\rho_1}^0 s^2 h_0(s)\tilde{g}(s)ds$
and thus,
$$ u(\rho)=c_0 h_0(\rho)+h_0(\rho)\int_{\rho_0}^\rho s^2 h_1(s)\tilde{g}(s)ds-h_1(\rho)
\int_0^\rho s^2 h_0(s)\tilde{g}(s)ds. $$
Similarly, the existence of $\lim_{\rho \to 1-}u(\rho)$ yields $\int_0^1 s^2 h_0(s)\tilde{g}(s)ds=0$
since $h_1$ is in $L^1$ near $\rho=1$.
This, however, is impossible since $s^2 h_0(s)\tilde{g}(s)>0$ for $s\in (0,1)$.
Consequently, there cannot exist a $\mb{u} \in \mc{D}(\mb{L})$ such that $(1-\mb{L})\mb{u}=\mb{g}$ and we
must have $m=1$.
\end{proof}

\subsection{The linear time evolution restricted to the stable subspace}

As already mentioned several times, the unstable eigenvalue $1 \in \sigma_p(\mb{L})$ stems from a
symmetry mode and does not correspond to a ``real'' instability.
Consequently, we consider the linear time evolution on the \emph{stable subspace} $\mc{N}:=
\rg(1-\mb{P})=\ker \mb{P}$ where $\mb{P}$ is the spectral projection defined in Eq.~\eqref{eq:P}.
Our aim is to derive a decay estimate for the subspace semigroup $\mb{S}(\tau)|_\mc{N}$.
To this end it is useful to recall the definition of the spectral bound of a closed operator.

\begin{definition}
Let $A: \mc{D}(A)\subset X \to X$ be a closed operator on a Banach space $X$.
Then the spectral bound $s(A)$ is defined as
$$ s(A):=\sup\{\Re \lambda: \lambda \in \sigma(A)\}. $$
\end{definition}

As before, we denote by $\mb{L}_\mc{N}$ the part of $\mb{L}$ in $\mc{N}$ and recall that
$\sigma(\mb{L}_\mc{N})=\sigma(\mb{L})\backslash\{1\}$ (see \cite{K80}).
Note that according to numerics \cite{BC05} we have in fact $s(\mb{L}_\mc{N})\approx -0.59$ but we 
emphasize that this information is not needed for the linear theory we are currently developing. We
need a preparing result which is interested in its own right as it shows that there 
do not exist unstable eigenvalues far away from 
the real axis.
We remark that this statement does not depend on the special form of the potential $V$. It is
merely a consequence of the structure of the differential operator $\mb{L}$.

\begin{lemma}
\label{lem:nospec}
Set $H_a:=\{z \in \mathbb{C}: \Re z \geq a\}$, $a \in \mathbb{R}$.
For any $\varepsilon>0$ there exist constants $C_1, C_2>0$ such that 
$$ \|\mb{R}_{\mb{L}}(\lambda)\|\leq C_2 $$
for all $\lambda \in H_{-\frac32+\varepsilon}$ with $|\lambda|\geq C_1$.
In particular, $\mb{L}$ does not have unstable eigenvalues far away from the real axis.
\end{lemma}

\begin{proof}
Let $\lambda \in H_{-\frac32+\varepsilon}$ for a fixed but arbitrary $\varepsilon>0$.
The identity $\lambda-\mb{L}=[1-\mb{L}'\mb{R}_{\mb{L}_0}(\lambda)](\lambda-\mb{L}_0)$ 
shows that $\lambda-\mb{L}$ is invertible
if and only if $1-\mb{L}'\mb{R}_{\mb{L}_0}(\lambda)$ is invertible.
Thus, we have to estimate 
$$ \mb{L}'\mb{R}_{\mb{L}_0}(\lambda)\mb{f}=\left (
\begin{array}{c}
-VK^2 [\mb{R}_{\mb{L}_0}(\lambda)\mb{f}]_2 \\ 0 \end{array} \right ) $$
for $\mb{f} \in \mc{H}$.
We write $\mb{u}=\mb{R}_{\mb{L}_0}(\lambda)\mb{f}$ and thus, 
$(\lambda-\mb{L}_0)\mb{u}=\mb{f}$. The second component
of this equation implies
$$ u_1(\rho)=\int_0^\rho s^3 u_2(s)ds+(\lambda-1)K^2 u_2(\rho)-K^2 f_2(\rho) $$
or, in other words,
$$ [\mb{R}_{\mb{L}_0}(\lambda)\mb{f}]_1(\rho)=\int_0^\rho s^3 
[\mb{R}_{\mb{L}_0}(\lambda)\mb{f}]_2(s)ds+(\lambda-1)
K^2[\mb{R}_{\mb{L}_0}(\lambda)\mb{f}]_2(\rho)-K^2 f_2(\rho). $$
From this we obtain the estimate
$$ \|K^2[\mb{R}_{\mb{L}_0}(\lambda)\mb{f}]_2\|_1\lesssim \tfrac{1}{|\lambda-1|}\|\mb{f}\| $$
by noting that $\|\mb{R}_{\mb{L}_0}(\lambda)\|\leq \frac{1}{\Re\lambda +\frac32}$ 
(Lemma \ref{lem:WPfree} and \cite{EN00}, p.~55, Theorem 1.10) where $\|\cdot\|_j$, $j=1,2$, denotes 
the norm on $\mc{H}_j$.
As a consequence, if $|\lambda|$ is sufficiently large, the Neumann series
$$ [1-\mb{L}'\mb{R}_{\mb{L}_0}(\lambda)]^{-1}=\sum_{k=0}^\infty [\mb{L}'\mb{R}_{\mb{L}_0}(\lambda)]^k $$
converges in norm and the claim follows.
\end{proof}

To conclude the linear perturbation theory, we estimate the linear evolution on the stable
subspace depending on the spectral bound of its generator.

\begin{proposition}
\label{prop:linstab}
Let $\varepsilon>0$ and set $\omega:=\max\{-\frac32, s(\mb{L}_\mc{N})\}+\varepsilon$.
Then there exists a constant $C_\varepsilon>0$ such that the semigroup
$\mb{S}(\tau)$ given in Corollary \ref{cor:WPlin} satisfies the estimate
$$ \|\mb{S}(\tau)(1-\mb{P})\|\leq C_\varepsilon e^{\omega \tau} $$
for all $\tau \geq 0$ where $\mb{P}$ is the spectral projection defined in Eq.~\eqref{eq:P}.
\end{proposition}

\begin{proof}
The operator $\mb{L}_\mc{N}$ is the generator of the subspace semigroup $\mb{S}(\tau)|_\mc{N}=
\mb{S}(\tau)(1-\mb{P})$ and its resolvent is given by $\mb{R}_\mb{L}(\lambda)|_\mc{N}$.
Consequently, the claimed estimate for $\mb{S}(\tau)(1-\mb{P})$ follows from the uniform boundedness of 
$\mb{R}_\mb{L}(\lambda)$ in the half--space $H_{\omega}$ (Lemma \ref{lem:nospec}) and
the Gearhart--Pr\"uss--Greiner Theorem (\cite{EN00},
p.~302, Theorem 1.11).
\end{proof}

We remark that if the self--similar solution $\psi^T$ is mode stable (cf.~Definition \ref{def:modstab})
then Lemma \ref{lem:specmodstab} implies $s(\mb{L}_\mc{N})<0$.
Hence, Proposition \ref{prop:linstab} shows that mode stability of $\psi^T$ 
implies linear stability. 
The numerically obtained value $s(\mb{L}_\mc{N})\approx -0.5889$ \cite{BC05} yields
the exponential decay
$$ \|\mb{S}(\tau)(1-\mb{P})\|\lesssim e^{-0.58 \,\tau},\quad \tau\geq 0 $$
for the linearized time evolution of perturbations of $\psi^T$.

\section{Nonlinear perturbation theory}
Based on Section \ref{sec:lin} we are now ready to treat the full system Eq.~\eqref{eq:phi1st}.
\emph{From now on we assume that $s(\mb{L}_\mc{N})<0$, i.e., that $\psi^T$ is mode stable.}
Proposition \ref{prop:linstab} then shows that the linearized time evolution on the stable 
subspace decays exponentially.
This puts us in an extremely convenient position since normally, at least in the 
study of wave equations,
one can at most hope for polynomial decay due to the continuous spectrum of the Laplacian.
In fact, as is well-known from dynamical systems theory, exponential decay of the linearization
carries over to the nonlinear evolution via Duhamel's formula.
In the PDE context one is of course faced with the additional complication that one needs good
mapping properties of the nonlinearity with respect to the spaces defined by the linear problem.
However, the norm we are using controls three derivatives and we are dealing with
a 5-dimensional problem where $\frac52+$ derivatives already suffice for a Moser estimate.
It is thus not surprising that we are able to obtain a Lipschitz property of the 
nonlinearity which is necessary to run a fixed point argument.
However, the presence of the symmetry mode $\mb{g}$ renders the linear evolution unstable and 
we have to overcome this by restricting ourselves to special initial data that live on a codimension
one ``manifold''. 
In a second step we then remove this restriction by adjusting the blowup time $T$.
At this point the role of the symmetry mode and its connection to the time translation 
invariance of the problem become evident.
As a matter of fact, the symmetry mode $\mb{g}$ corresponds to 
the derivative at $T=1$ of the curve $T \mapsto (\psi^T(0,\cdot), \psi_t^T(0,\cdot))$ in the 
space of initial data.

\subsection{Estimates for the nonlinearity}
As in the proof of Lemma \ref{lem:compact} we write $\mc H=\mc H_1 \times \mc H_2$ and denote
by $\|\cdot\|_j$ the respective norm on $\mc{H}_j$, $j=1,2$.
As a reminder we recall that
\[ \|u\|_1^2=\int_0^1 |D^2 u(\rho)|^2d\rho,\quad \|u\|_2^2=\int_0^1 |u'(\rho)|^2 d\rho \]
where $Du(\rho)=\frac{1}{\rho}u'(\rho)$.
Furthermore, from now on we assume all functions to be real--valued.
The nonlinear term in Eq.~\eqref{eq:phi1st} reads $\rho N_T(\frac{1}{\rho^3}K^2 \phi_2)$. 
In fact, since we switched to similarity coordinates, we have by Eq.~\eqref{eq:nl},
\[ N_T(u)(\rho)=9W_0(\rho)u(\rho)^2+3u(\rho)^3 \]
and thus, the nonlinear term is independent of $T$.
For the following it is useful to separate the functional dependence more clearly and we therefore
define
\begin{equation}
\label{eq:N}
\tilde{N}(x,\rho):=9W_0(\rho)x^2+3x^3.
\end{equation}
Since the nonlinear term occurs in the first component of Eq.~\eqref{eq:phi1st} and takes an argument
from the second component, we have to study its mapping properties as a map from $\mc H_2$ to
$\mc H_1$.
We start by defining an auxiliary operator
\[ Au(\rho):=\frac{1}{\rho^3}K^2 u(\rho) \]
which represents the argument of the nonlinearity where, as always, $Ku(\rho)=\int_0^\rho su(s)ds$.

\begin{lemma}
\label{lem:A}
We have the bounds 
\begin{align*} 
\|(\cdot)^{-2}Au\|_{L^2(0,1)}&\lesssim \|u\|_2 \\
\|(\cdot)^{-\frac32}Au\|_{L^\infty(0,1)}&\lesssim \|u\|_2 
\end{align*}
for all $u\in\mc H_2$.
\end{lemma}

\begin{proof}
By Lemma \ref{lem:basicH} we may assume $u\in C_c^\infty(0,1]$ and thus,
\[ \|(\cdot)^{-2}Au\|_{L^2(0,1)}^2=\int_0^1 \frac{|Au(\rho)|^2}{\rho^4}d\rho=\int_0^1 \frac{|K^2 u(\rho)|^2}{\rho^{10}}d\rho 
\lesssim \int_0^1 |u'(\rho)|^2 d\rho \]
by repeated application of Hardy's inequality (Lemma \ref{lem:Hardy}). For the second bound we note that
\begin{align*}
|Au(\rho)|&\leq\frac{1}{\rho^3}\int_0^\rho s \int_0^s t |u(t)|dt ds \\
&\leq \frac{1}{\rho^3}\int_0^\rho s \left (\int_0^s t^4 dt\right )^{1/2}
\left (\int_0^s \frac{|u(t)|^2}{t^2}dt \right )^{1/2}ds \\
&\lesssim \rho^\frac32 \|u'\|_{L^2(0,1)}
\end{align*}
again by Hardy's inequality.
\end{proof}

We provide similar bounds for the derivatives of $A$.

\begin{lemma}
\label{lem:DA}
We have the bounds
\begin{align*}
\|D Au\|_{L^2(0,1)}&\lesssim \|u\|_2 \\
\|(\cdot)^2 D^2 Au \|_{L^2(0,1)}&\lesssim \|u\|_2 \\
\|(\cdot)^\frac12 DAu\|_{L^\infty(0,1)}&\lesssim \|u\|_2 \\
\|(\cdot)^\frac52 D^2 Au\|_{L^\infty(0,1)}&\lesssim \|u\|_2
\end{align*}
for all $u\in \mc H_2$.
\end{lemma}

\begin{proof}
The proof consists of straightforward applications of Hardy's inequality, the logic
being, of course, that each application of $D$ loses two powers of $\rho$.
\end{proof}

Now we define $N(u)(\rho):=\rho\tilde{N}(Au(\rho),\rho)$ which corresponds to the nonlinearity
in Eq.~\eqref{eq:phi1st}.
We have the following crucial result which is key to control the nonlinearity.

\begin{lemma}
\label{lem:N}
The function $N$ maps $\mc{H}_2$ to $\mc{H}_1$ and we have the bound
\[ \|N(u)-N(v)\|_1\lesssim (\|u\|_2+\|v\|_2)\|u-v\|_2 \]
for all $u,v \in \mc B_2$, the open unit ball in $\mc H_2$.
Furthermore, $N(0)=0$ and $N$ is Fr\'echet differentiable at $0$ with derivative
$D^F N(0)=0$.
\end{lemma}

\begin{proof}
Evidently, we have 
\[ \tilde{N}(x,\rho)-\tilde{N}(y,\rho)=[9W_0(\rho)(x+y)+3(x^2+xy+y^2)](x-y) \]
for all $x,y \in \R$ and $\rho\in [0,1]$ and thus, we obtain
\begin{align*}
D^2 [N(Au)-N(Av)]&=D^2 [9(\cdot)W_0(Au+Av)+3(\cdot)((Au)^2+Au Av+(Av)^2)](Au-Av) \\
&\quad +2D[9(\cdot)W_0(Au+Av)+3(\cdot)((Au)^2+AuAv+(Av)^2)]D(Au-Av) \\
&\quad +[9(\cdot)W_0(Au+Av)+3(\cdot)((Au)^2+AuAv+(Av)^2)]D^2(Au-Av).
\end{align*}
We need to put this whole expression in $L^2$ and therefore, we place the terms involving $Au-Av$ in $L^2$ and
the rest in $L^\infty$ and apply Lemmas \ref{lem:A} and \ref{lem:DA} to bound them.
For instance, we use
\begin{align*}
\|D^2[(\cdot)Au](Au-Av)\|_{L^2(0,1)}&\leq \|(\cdot)^2 D^2 [(\cdot)Au]\|_{L^\infty(0,1)}
\|(\cdot)^{-2}(Au-Av)\|_{L^2(0,1)} \\
&\lesssim \|u\|_2 \|u-v\|_2
\end{align*}
since
\begin{align*} \|(\cdot)^2 D^2 [(\cdot)Au]\|_{L^\infty(0,1)}&\lesssim \|(\cdot)^{-1}Au\|_{L^\infty(0,1)} 
+\|(\cdot)DAu\|_{L^\infty(0,1)} \\
&\quad + \|(\cdot)^3 D^2 Au\|_{L^\infty(0,1)} \\
&\lesssim  \|u\|_2.
\end{align*}
Similarly, we obtain
\begin{align*}
\|D^2[(\cdot)(Au)^2](Au-Av)\|_{L^2(0,1)}&\leq \|(\cdot)^2 D^2[(\cdot)(Au)^2]\|_{L^\infty(0,1)}\|(\cdot)^{-2}
(Au-Av)\|_{L^2(0,1)} \\
&\lesssim \|u\|_2^2 \|u-v\|_2
\end{align*}
since, e.g.
\begin{align*}
\|(\cdot)^3 D^2 (Au)^2\|_{L^\infty(0,1)}&\lesssim \|(\cdot)^3 (DAu)^2\|_{L^\infty(0,1)}+\|(\cdot)^3 AuD^2Au\|_{L^\infty(0,1)} \\
&\lesssim \|u\|_2^2,
\end{align*}
etc. The other terms can be treated in the exact same fashion.
Since $\tilde{N}(0,\rho)=0$ for all $\rho\in [0,1]$, we obtain $N(0)=0$ and thus, 
$\|N(u)\|_1\lesssim \|u\|_2^2$ for all $u\in\mc B_2$.
This estimate also implies the Fr\'echet differentiability of $N$ at $0$ with $D^F N(0)=0$.
\end{proof}

In order to write the main equation \eqref{eq:phi1st} as an ordinary differential equation
on the Hilbert space $\mc H$ we introduce the vector--valued nonlinearity 
\[ \mb N(\mb u):=\left (\begin{array}{c}-N(u_2) \\ 0 \end{array} \right ) \]
for $\mb u=(u_1,u_2)$.

\begin{lemma}
\label{lem:vecN}
The nonlinearity $\mb{N}$ maps $\mc H$ to itself and satisfies the bound
\[ \|\mb{N}(\mb u)-\mb N(\mb v)\|\lesssim (\|\mb u\|+\|\mb v\|)\|\mb u-\mb v\| \]
for all $\mb u, \mb v \in \mc B$ where $\mc B$ denotes the open unit ball in $\mc H$.
Furthermore, $\mb N(\mb 0)=\mb 0$ and $\mb N$ is Fr\'echet differentiable at $\mb 0$ with
$D^F \mb N(\mb 0)=\mb 0$.
\end{lemma}

\begin{proof}
Since $\|\mb N(\mb u)\|=\|N(u_2)\|_1$, the statement is an immediate consequence of Lemma \ref{lem:N}.
\end{proof}

Consequently, Eq.~\eqref{eq:phi1st} can be written as
\begin{equation}
\label{eq:odeH}
\tfrac{d}{d\tau}\Phi(\tau)=\mb L\Phi(\tau)+\mb N(\Phi(\tau))
\end{equation}
and our aim is to study Eq.~\eqref{eq:odeH} with small Cauchy data $\mb u\in \mc H$ 
prescribed at $\tau=-\log T$.
Thus, by Duhamel's formula we may rewrite the problem as
\begin{equation}
\label{eq:Duhamel}
\Phi(\tau)=\mb S(\tau+\log T)\mb u+\int_{-\log T}^\tau \mb S(\tau-\tau')\mb N(\Phi(\tau'))d\tau'
\end{equation}
which is equivalent to
\begin{equation}
\label{eq:Duhamel2}
\Psi(\tau)=\mb S(\tau)\mb u+\int_0^\tau \mb S(\tau-\tau')\mb N(\Psi(\tau'))d\tau'
\end{equation}
for $\Psi(\tau)=\Phi(\tau-\log T)$.

\subsection{Existence for codimension one data via the Lyapunov--Perron method}

Our goal is to prove global existence \footnote{Recall that global existence in the variable $\tau$
really means local existence for the original equation in the backward lightcone $\mc C_T$.} 
for Eq.~\eqref{eq:Duhamel2}. 
This is not straightforward since the linear time evolution is unstable due to the presence
of the symmetry mode $\mb g$ given in Eq.~\eqref{eq:g}.
Indeed, we have $\mb S(\tau)\mb g=e^\tau \mb g$.
Consequently, in a first step we modify Eq.~\eqref{eq:Duhamel2} and consider
\begin{equation}
\label{eq:Duhamel3}
\Psi(\tau)=\mb S(\tau)(1-\mb P)\mb u-\int_0^\infty e^{\tau-\tau'}\mb P \mb N(\Psi(\tau'))d\tau'
+\int_0^\tau \mb S(\tau-\tau')\mb N(\Psi(\tau'))\tau'
\end{equation}
instead.
Comparison with Eq.~\eqref{eq:Duhamel2} shows that we have actually modified the \emph{initial data}
by subtracting the term 
\begin{equation}
\label{eq:correction}
\mb P \left [\mb u+\int_0^\infty e^{-\tau'}\mb N(\Psi(\tau'))d\tau' \right ] 
\end{equation}
which is an element of the unstable subspace $\langle \mb g \rangle$.
However, note carefully that the modification 
\emph{depends on the solution itself}.
As we will see, this modification stabilizes the evolution and we are able to obtain global
existence.
The procedure of modifying the data in order to force stability of the evolution is known as the
Lyapunov--Perron method in (finite dimensional) dynamical systems theory.
In a second step we then show how to obtain a solution of Eq.~\eqref{eq:Duhamel2}.

In order to be able to apply a fixed point argument, we define an operator $\mb K$ by
\begin{equation}
\label{eq:defK}
\mb K(\Psi;\mb u)(\tau):=\mb S(\tau)(1-\mb P)\mb u-\int_0^\infty e^{\tau-\tau'}\mb P \mb N(\Psi(\tau'))d\tau'
+\int_0^\tau \mb S(\tau-\tau')\mb N(\Psi(\tau'))\tau'.
\end{equation}
As a consequence, fixed points of $\mb K(\cdot;\mb u)$ 
correspond to solutions of Eq.~\eqref{eq:Duhamel3}.
We run the fixed point argument in a Banach space $\mc X$ defined by
\[ \mc X:=\{\Psi \in C([0,\infty),\mc H): \sup_{\tau>0}e^{|\omega| \tau}\|\Psi(\tau)\|<\infty \} \]
where $\omega$ is from Proposition \ref{prop:linstab}, i.e., after fixing a small $\varepsilon>0$,
the linear evolution satisfies
$\|\mb S(\tau)(1-\mb P)\|\lesssim e^{\omega\tau}$ with $\omega=\max\{-\frac32,s(\mb L_{\mc N})\}+\varepsilon$ 
and $\omega<0$ by the assumed mode stability.
We also write
\[ \|\Psi\|_{\mc X}:=\sup_{\tau>0}e^{|\omega| \tau}\|\Psi(\tau)\| \]
for the norm on $\mc X$.
Furthermore, we denote by $\mc X_\delta \subset \mc X$ the closed subset defined by
\[ \mc X_\delta:=\{\Psi \in \mc X: \|\Psi\|_{\mc X}\leq \delta\}. \]
In other words, $\mc X_\delta$ is the closed $\delta$--ball in $\mc X$.

\begin{lemma}
\label{lem:fixedpoint}
Let $\delta>0$ be sufficiently small and assume $\|\mb u\|\leq \delta^2$.
Then the operator $\mb K$ maps $\mc X_\delta$ to itself and is contractive, i.e., 
\[ \|\mb K(\Psi; \mb u)-\mb K(\Phi; \mb u)\|_{\mc X}\leq \tfrac12 \|\Psi-\Phi\|_{\mc X} \]
for all $\Psi,\Phi \in \mc X_\delta$.
As a consequence, there exists a unique fixed point of $\mb K(\cdot;\mb u)$ in $\mc X_\delta$.
\end{lemma}

\begin{proof}
Note first that $\mb K(\Psi;\mb u) \in C([0,\infty),\mc H)$ for any $\Psi \in \mc X$ and $\mb u \in \mc H$
by the strong continuity of the semigroup $\mb S$.
We decompose the operator $\mb K$ according to
\[ \mb K(\Psi;\mb u)(\tau)=\mb P \mb K(\Psi; \mb u)(\tau)+(1-\mb P)\mb K(\Psi; \mb u)(\tau). \]
By using Lemma \ref{lem:vecN} and Proposition \ref{prop:linstab} we readily estimate
\begin{align*}
\|\mb P \mb K(\Psi; \mb u)(\tau)\|&\leq \int_\tau^\infty e^{\tau-\tau'}\mb \|\mb P \mb N(\Psi(\tau'))\|d\tau'
\lesssim \sup_{\tau'>0}e^{2|\omega|\tau'}\|\Psi(\tau')\|^2\int_\tau^\infty e^{\tau-(1+2|\omega|)\tau'}d\tau' \\
&\lesssim \delta^2 e^{-2|\omega|\tau}
\end{align*}
as well as
\begin{align*}
\|(1-\mb P) \mb K(\Psi; \mb u)\|&\lesssim e^{-|\omega|\tau}\|\mb u\|+\int_0^\tau \|\mb S(\tau-\tau')(1-\mb P)\mb N(\Psi(\tau'))\|d\tau' \\
&\lesssim \delta^2 e^{-|\omega|\tau}+\int_0^\tau e^{-|\omega|(\tau-\tau')}\|\Psi(\tau')\|^2 d\tau' \\
&\lesssim \delta^2 e^{-|\omega|\tau}
\end{align*}
and this yields $\mb K(\Psi;\mb u) \subset \mc X_\delta$ for all $\Psi \in \mc X_\delta$ provided
$\|\mb u\|\leq \delta^2$.
By a completely analogous computation we obtain the estimates
\begin{align*}
\|\mb P \mb K(\Psi; \mb u)(\tau)-\mb P \mb K(\Phi; \mb u)(\tau)\|&\lesssim \delta e^{-|\omega|\tau} \|\Psi-\Phi\|_{\mc X} \\
\|\mb (1-\mb P) \mb K(\Psi; \mb u)(\tau)-\mb (1-\mb P) \mb K(\Phi; \mb u)(\tau)\|&\lesssim \delta e^{-|\omega|\tau} \|\Psi-\Phi\|_{\mc X}
\end{align*}
which imply the claimed contraction property provided $\delta>0$ is sufficiently small.
Consequently, the contraction mapping principle yields the existence of a unique fixed point in $\mc X_\delta$.
\end{proof}

We obtain a global solution of the modified problem Eq.~\eqref{eq:Duhamel3} with small data.

\begin{proposition}
\label{prop:solmod}
Let $\mc U\subset \mc H$ be a sufficiently small open ball with center $\mb 0$ in $\mc H$.
Then, for any given $\mb u \in \mc U$, there exists a unique solution
$\mb \Psi(\mb u) \in \mc X_\delta$ of Eq.~\eqref{eq:Duhamel3}.
Furthermore, the map $\mb \Psi: \mc U\subset \mc H \to \mc X$ is continuous and Fr\'echet differentiable
at $\mb 0$.
\end{proposition}

\begin{proof}
The existence of $\mb \Psi(\mb u)$ is a consequence of Lemma \ref{lem:fixedpoint}.
Now note that
\begin{align*}
\|\mb{\Psi}(\mb u)-\mb \Psi(\mb v)\|_{\mc X}&\leq \|\mb K(\mb \Psi(\mb u);\mb u)-\mb K(\mb \Psi(\mb v);\mb u)\|_{\mc X}
+\|\mb K(\mb \Psi(\mb v);\mb u)-\mb K(\mb \Psi(\mb v);\mb v)\|_{\mc X} \\
&\leq \tfrac12 \|\mb \Psi(\mb u)-\mb \Psi(\mb v)\|_{\mc X}+\|\mb K(\mb \Psi(\mb v);\mb u)-\mb K(\mb \Psi(\mb v);\mb v)\|_{\mc X}
\end{align*}
for all $\mb u, \mb v \in \mc U$ by Lemma \ref{lem:fixedpoint} and thus,
$\|\mb{\Psi}(\mb u)-\mb \Psi(\mb v)\|_{\mc X}\leq 2 \|\mb K(\mb \Psi(\mb v);\mb u)-\mb K(\mb \Psi(\mb v);\mb v)\|_{\mc X}$.
By definition of $\mb K$ and Proposition \ref{prop:linstab} we have
\begin{align*}
\|\mb K(\mb \Psi(\mb v);\mb u)(\tau)-\mb K(\mb \Psi(\mb v);\mb v)(\tau)\|&=\|\mb S(\tau)(1-\mb P)(\mb u-\mb v)\| \\
&\lesssim e^{-|\omega|\tau}\|\mb u-\mb v\|
\end{align*}
which yields $\|\mb K(\mb \Psi(\mb v);\mb u)-\mb K(\mb \Psi(\mb v);\mb v)\|_{\mc X}\lesssim \|\mb u-\mb v\|$ and we conclude
\[ \|\mb{\Psi}(\mb u)-\mb \Psi(\mb v)\|_{\mc X}\lesssim \|\mb u-\mb v\| \]
for all $\mb u, \mb v \in \mc U$.
Hence, $\mb \Psi$ is Lipschitz continuous.
Furthermore, we claim that $[D^F \mb \Psi(\mb 0)\mb u](\tau)=\mb S(\tau)(1-\mb P)\mb u$.
Indeed, since $\mb \Psi(\mb 0)=\mb 0$, we obtain
\begin{align*}
\mb P [\mb \Psi(\mb u)(\tau)-\mb \Psi (\mb 0)(\tau)-\mb S(\tau)(1-\mb P)\mb u]&=-\int_\tau^\infty e^{\tau-\tau'}\mb P \mb N(\mb \Psi(\mb u)(\tau'))d\tau' \\
(1-\mb P) [\mb \Psi(\mb u)(\tau)-\mb \Psi (\mb 0)(\tau)-\mb S(\tau)(1-\mb P)\mb u]&=\int_0^\tau \mb S(\tau-\tau')(1-\mb P)\mb N(\mb \Psi(\mb u)(\tau'))d\tau' 
\end{align*}
and as in the proof of Lemma \ref{lem:fixedpoint} this implies
\[ \|\mb \Psi(\mb u)(\tau)-\mb \Psi (\mb 0)(\tau)-\mb S(\tau)(1-\mb P)\mb u 
\|\lesssim e^{-|\omega|\tau}\|\mb \Psi(\mb u)\|_{\mc X}^2\lesssim e^{-|\omega|\tau}\|\mb u\|^2 \]
by the above and the claim follows.
\end{proof}

\subsection{Existence for general small data}
\label{sec:globalex}

In this section we construct a global solution to Eq.~\eqref{eq:Duhamel2}.
This amounts to removing the modification which led from Eq.~\eqref{eq:Duhamel2} to Eq.~\eqref{eq:Duhamel3}.
Recall that we had to introduce this modification because of the instability of the linear evolution
and this instability emerges from the time translation invariance of the original equation.
Consequently, one should be able to remove the instability by shifting the blowup time. Mathematically, this manifests
itself in the fact that the modification turns out to be identically zero once we have chosen the correct blowup time $T$.

Since the translated equation \eqref{eq:Duhamel2} is independent of $T$, the only place where
the blowup time enters is in the data.
For given data $(f,g)$ as in Eq.~\eqref{eq:mainCauchy}, 
we denote by $\mb v$ the corresponding data in the new coordinates relative
to $\psi^1$, the fundamental self--similar solution with blowup time $T=1$.
Explicitly, we have
\begin{equation}
\label{eq:v}
\mb v(\rho)=\left ( \begin{array}{c}
\rho^3 [g(\rho)-\psi_t^1(0,\rho)] \\ \mc D^2 [f-\psi^1(0,\cdot)](\rho) \end{array} \right ), 
\end{equation}
cf.~Eq.~\eqref{eq:phi1stdata} and recall that $\mc D^2 f(r)=rf''(r)+5 f'(r)+\frac{3}{r}f(r)$.
Then we set
\begin{equation}
\label{eq:defU}
\begin{aligned}
\mb U(\mb v,T)(\rho):&=\left ( \begin{array}{c}
\tfrac{1}{T^2} v_1(T\rho) \\ Tv_2(T\rho) \end{array} \right )
+\left (
\begin{array}{c}
T \rho^3[ \psi^1_t(0,T\rho)-\psi^T_t(0,T\rho)] \\
T \mc D^2 [\psi^1(0,\cdot)-\psi^T(0,\cdot)](T\rho)
\end{array} \right ) \\
&=\left ( \begin{array}{c} T\rho^3[g(T\rho)-\psi_t^T(0,T\rho)] \\
T\mc D^2 [f-\psi^T(0,\cdot)](T\rho) \end{array} \right ).
\end{aligned}
\end{equation}
Thus, with Eq.~\eqref{eq:v}, the initial data for Eq.~\eqref{eq:odeH} can be written as
\[ \Psi(0)=\Phi(-\log T)=\mb U(\mb v,T), \]
see Eq.~\eqref{eq:phi1stdata}.
The point of this notation is, of course, that $\mb v$ is independent of $T$ and thus,
the functional dependence of the data $\Psi(0)$ on $(f,g)$ and $T$ is now explicit.
We also remark that the data $(f,g)$ have to be prescribed on the interval $[0,T]$ but $T$
is not known in advance. 
However, this defect is easily remedied by simply prescribing the data on $[0,\frac32]$ since we may 
always assume that $T \in I:=(\frac12,\frac32)$ by the perturbative character of our construction.
Consequently, we set
\[ \|\mb u\|_{\hat{\mc H}}^2:=\int_0^\frac32 |D^2 u_1(\rho)|^2 d\rho+\int_0^\frac32 |u_2'(\rho)|^2d\rho, \]
and denote the respective Hilbert space by $\hat{\mc H}$.
Note carefully that in view of Eq.~\eqref{eq:v} we have
\[ \|\mb v\|_{\hat{\mc H}}=\|(f,g)-(\psi^1(0,\cdot),\psi_t^1(0,\cdot))\|_{\mc{E}(\frac32)} \]
with the $\mc{E}(\frac32)$--norm from Theorem \ref{thm:main}.

\begin{lemma}
\label{lem:U}
The function $\mb U$ maps $\hat{\mc H}\times I$ to $\mc H$ continuously and $\mb U(\mb 0,1)=\mb 0$. Furthermore, 
$\mb U(\mb 0,\cdot): I \to \mc H$ is Fr\'echet differentiable with partial derivative
\[ D_2^F \mb U(\mb 0,1)\lambda=-240 \lambda \mb g \]
for all $\lambda \in \R$ where $\mb g$ is the symmetry mode given in Eq.~\eqref{eq:g}.
\end{lemma}

\begin{proof}
It follows immediately from the definition of $\mb U$ that
\[ \|\mb U(\mb v,T)-\mb U(\mb w,T)\|\lesssim \|\mb v-\mb w\|_{\hat{\mc H}} \]
for all $\mb v, \mb w \in \hat{\mc H}$, uniformly in $T \in I$.
Thus, in order to show continuity of $\mb U$, it suffices to prove continuity of $\mb U(\mb v,\cdot): I\to \mc H$ for
fixed $\mb v \in \hat{\mc H}$.
By Lemma \ref{lem:basicH} we may assume $v_j \in C[0,\frac32]$, $j=1,2$.
Furthermore, recall that $\psi^T(0,r)=W_0(\frac{r}{T})-1$ and $\psi_t^T(0,r)=\frac{r}{T^2}W_0'(\frac{r}{T})$ with
\[ W_0(\rho)=\frac{1-\rho^2}{1+\frac35 \rho^2}. \]
This implies $\mc D^2 \psi^T(0,\cdot) \in C^\infty[0,\frac32]$ and continuity of $\mb U(\mb v,\cdot)$ follows
by the continuity of $T \mapsto \|f(T\cdot)\|_{L^2(0,1)}: I\to\R$ for $f \in C[0,\frac32]$.
Obviously, we have $\mb U(\mb 0,1)=\mb 0$ and the Fr\'echet differentiability of $\mb U(\mb 0,\cdot)$
on $I$ is also evident.
By straightforward differentiation we obtain
$D^F_2 \mb U(\mb 0,1)\lambda=-240 \lambda \mb g$
as claimed.
\end{proof}

By Lemma \ref{lem:U} we infer that $\mb U(\mb v,T) \in \mc H$ is small provided $\mb v$ is sufficiently
small in $\hat{\mc H}$ and $T$ is sufficiently close to $1$.
Consequently, we obtain $\mb U(\mb v,T) \in \mc U$ where $\mc U$ is from 
Proposition \ref{prop:solmod} and there exists a solution $\mb \Psi(\mb U(\mb v,T)) \in \mc X$ of Eq.~\eqref{eq:Duhamel3}
with initial data $\mb U(\mb v,T)$.
The correction term \eqref{eq:correction}, which was introduced to suppress the instability of the linear evolution, is given by
\[ \mb P\left [\mb U(\mb v,T)+\int_0^\infty e^{-\tau'}\mb N(\mb \Psi(\mb U(\mb v,T))(\tau'))d\tau' \right ]=:\mb F(\mb v,T) \]
and by the above considerations, $\mb F: \mc V\times J \to \langle \mb g \rangle$ is a well-defined map
for $\mc V$ a sufficiently small open ball around $\mb 0$ in $\hat{\mc H}$ and $J \subset I$ a
sufficiently small open interval with $1 \in J$.
If $\mb F(\mb v,T)=\mb 0$ then the correction term vanishes and $\mb \Psi(\mb U(\mb v,T))$ is also a solution
to the original equation Eq.~\eqref{eq:Duhamel2}.
Obviously, we have $\mb F(\mb 0,1)=\mb 0$ since $\mb U(\mb 0,1)=\mb 0$ and the corresponding solution
is $\mb \Psi(\mb 0)=\mb 0$ by the uniqueness in $\mc X_\delta$ (Proposition \ref{prop:solmod}).
Now we show that for any small $\mb v$ we can find a $T$ such that $\mb F(\mb v,T)=\mb 0$.
We need one additional technical result.

\begin{lemma}
\label{lem:F}
The mapping $\mb F: \mc V\times J\to \langle \mb g \rangle$ is continuous. Furthermore, $\mb F(\mb 0,\cdot): J\to \langle \mb g \rangle$
is Fr\'echet differentiable at $1$ with derivative
\[ D^F_2 \mb F(\mb 0,1)\lambda=-240  \lambda\mb g \]
for all $\lambda \in \R$.
\end{lemma}

\begin{proof}
It is convenient to introduce a symbol for the integral operator in the definition of $\mb F$ and we write
\[ \mb B \Psi:=\int_0^\infty e^{-\tau'}\Psi(\tau')d\tau'. \]
Then $\mb B: \mc X \to \mc H$ is linear and bounded.
Furthermore, we define $\hat{\mb N}: \mc X\to \mc X$ by $\hat{\mb N}(\Psi)(\tau):=\mb N(\Psi(\tau))$.
By Lemma \ref{lem:vecN} we have
\[ \|\hat{ \mb N}(\Psi)\|_{\mc X}=\sup_{\tau>0}e^{|\omega|\tau}\|\mb{N}(\Psi(\tau))\|
\lesssim \sup_{\tau>0}e^{|\omega|\tau}\|\Psi(\tau)\|^2 \leq \|\Psi\|_{\mc X}^2 \]
which shows that $\hat{\mb N}$ is Fr\'echet differentiable at $\mb 0$ with 
$\hat {\mb N}(\mb 0)=\mb 0$ and $D^F \hat{\mb N}(\mb 0)=\mb 0$.
Consequently, the function $\mb F$ can be written as
\[ \mb F(\mb v,T)=\mb P \left [\mb U(\mb v,T)+\mb B \hat{\mb N} (\mb \Psi (\mb U(\mb v,T))) \right ] \]
and by Lemmas \ref{lem:vecN}, \ref{lem:U} and Proposition \ref{prop:solmod} it follows that
$\mb F$ is continuous.
Furthermore, by the chain rule for Fr\'echet derivatives we immediately infer
\begin{align*} 
D^F_2 \mb F(\mb 0,1)&=\mb P D^F_2 \mb U(\mb 0,1)+\mb B D^F \hat {\mb N}(\mb 0)D^F \mb \Psi(\mb 0)D^F_2 \mb U(\mb 0,1) \\
&=\mb P D^F_2 \mb U(\mb 0,1)
\end{align*}
and Lemma \ref{lem:U} yields $D^F_2 \mb F(\mb 0,1)\lambda=-240 \lambda \mb g$ for all $\lambda \in \R$.
\end{proof}

\begin{lemma}
\label{lem:F0}
Let $\mc V \subset \hat{\mc H}$ be a sufficiently small open ball around $\mb 0$.
Then, for any $\mb v\in \mc V$, there exists a $T \in (\frac12,\frac32)$ such that $\mb F(\mb v,T)=\mb 0$.
\end{lemma}

\begin{proof}
Denote by $i: \langle \mb g \rangle \to \R$ the vector space isomorphism given by $i(\lambda \mb g)=\lambda$, $\lambda \in \R$,
and set $f:=i \circ \mb F$.
We have $f(\mb 0,1)=0$ and Lemma \ref{lem:F} shows that $f(\mb 0,\cdot): J \to \R$ is differentiable
at $1$ with $\partial_2 f(\mb 0,1)\not= 0$.
Consequently, we obtain $T_-, T_+ \in J$ such that $f(\mb 0,T_-)<0$ and $f(\mb 0,T_+)>0$.
Since $f: \mc V \times J \to \R$ is continuous, we find that $f(\mb v,T_-)<0$ and $f(\mb v,T_+)>0$
for all $\mb v \in \mc V$ provided $\mc V$ is sufficiently small.
Thus, by the intermediate value theorem there exists a $T \in (T_-,T_+)$ such that
$f(\mb v,T)=0$.
\end{proof}

We formulate the main result as a theorem.

\begin{theorem}
Let $\mc V \subset \hat{\mc H}$ be a sufficiently small ball with center $\mb 0$.
Then, for any $\mb v \in \mc V$, there exists a $T \in (\frac12,\frac32)$ such that
the Cauchy problem
\[ \left \{ \begin{array}{l}
\tfrac{d}{d\tau}\Phi(\tau)=\mb L \Phi(\tau)+\mb N(\Phi(\tau)) \\
\Phi(-\log T)=\mb U(\mb v,T)
\end{array} \right .
\]
has a unique mild solution $\Phi \in C([-\log T,\infty),\mc H)$ satisfying
\[ \|\Phi(\tau)\|\lesssim e^{-|\omega|\tau} \]
for all $\tau>-\log T$.
\end{theorem}

\begin{proof}
It only remains to prove that the solution is unique in $C([-\log T,\infty),\mc H)$. However, this
is a simple consequence of the fact that $\Phi$ is a fixed point of a contraction mapping.
\end{proof}

\bibliography{ym}{}
\bibliographystyle{plain}

\end{document}